\documentclass[10pt,letterpaper]{amsart}
\usepackage{amssymb,esint,empheq}
\numberwithin{equation}{section}
\theoremstyle{plain}
\newtheorem{theorem}[equation]{Theorem}
\newtheorem{lemma}[equation]{Lemma}
\newtheorem{proposition}[equation]{Proposition}
\newtheorem{corollary}[equation]{Corollary}
\theoremstyle{definition}
\newtheorem{definition}[equation]{Definition}

\newtheorem{remark}[equation]{Remark}

\newcommand*{\NN}{\mathbb{N}}
\newcommand*{\RR}{\mathbb{R}}
\newcommand*{\eps}{\varepsilon}
\newcommand*{\Om}{\Omega}
\providecommand*{\coloneq}{\mathbin{:=}}
\newcommand*{\dOm}{\partial\Omega}
\newcommand*{\dd}{{\mathrm d}}
\newcommand*{\loc}{{\mathrm{loc}}}
\newcommand*\lapl{\mathop{{}\Delta}\nolimits}
\DeclareMathOperator{\dist}{dist}
\DeclareMathOperator{\rad}{rad}
\DeclareMathOperator{\diam}{diam}
\begin{document}
\title[Neumann problem for $p$-Laplace equation in metric spaces]{Neumann problem
for $p$\,-Laplace equation\\in metric spaces using a variational approach:
existence, boundedness, and boundary regularity}
\author[{L.\@ Mal\'{y}\and N.\@ Shanmugalingam}]{Luk\'{a}\v{s} Mal\'{y}\and Nageswari Shanmugalingam}
\address{University of Cincinnati\\Department of Mathematical Sciences\\P.O.~Box~210025\\Cincinnati, OH 45221-0025\\USA}
\email{malyls@ucmail.uc.edu}
\email{shanmun@uc.edu}
\thanks{The research of L.M. was supported by the Knut and Alice Wallenberg Foundation,
and the research of N.S. was partially supported by the NSF grant~\#DMS-1500440
(U.S.). The authors thank John Lewis for pointing out the
reference~\cite{GaMPR}.}
\date{September 21, 2016}
\begin{abstract}
We employ a variational approach to study the Neumann boundary value problem
for the $p$-Laplacian on bounded smooth-enough domains in the metric setting,
and show that solutions exist and are bounded. The boundary data considered
are Borel measurable bounded functions. We also study boundary continuity
properties of the solutions. One of the key tools utilized is the trace
theorem for Newton-Sobolev functions, and another is an analog of the De
Giorgi inequality adapted to the Neumann problem.
\end{abstract}
\maketitle
\section{Overview}
Amongst the two types of boundary value problems in PDEs, Dirichlet and Neumann
problems, the Dirichlet problem is currently the most well-studied. In the
Euclidean setting, much of the research on Neumann boundary value problem
focused on the zero boundary value problem, the so-called natural boundary
value. The general Neumann boundary value problem for the $p$-Laplacian is the
following: find $u$ in the appropriate Sobolev class such that
\begin{empheq}[left=\empheqlbrace]{alignat=2}
  \notag-\lapl_p u & = 0&\quad& \text{in } \Om,\\
  \:-|\nabla u|^{p-2} \partial_\eta u & = f&& \text{on } \dOm,
  \label{eq:PDE}
\end{empheq}
which, in its weak formulation, would mean finding $u$ in the Sobolev class
$W^{1,p}(\Om)$ such that whenever $\varphi\in W^{1,p}(\Om)$,
\[
 \int_\Om |\nabla u(x)|^{p-2}\nabla u(x)\cdot\nabla\varphi(x)\, dx
 +\int_{\dOm}|\nabla u(\zeta)|^{p-2}\varphi(\zeta)\partial_\eta u(\zeta)d\mathcal{H}^{n-1}(\zeta)=0,
\]
that is,
\[
 \int_\Om |\nabla u(x)|^{p-2}\nabla u(x)\cdot\nabla\varphi(x)\, dx
 -\int_{\dOm}\varphi(\zeta)f(\zeta) \, d\mathcal{H}^{n-1}(\zeta)=0,
\]
where $\partial_\eta u(\zeta)$ is the directional derivative of $u$ in the
direction of the outer normal to $\dOm$ at $\zeta$. Its variational formulation
is to find $u$ that minimizes the functional
\begin{equation*}
  I(u) = \int_\Om |\nabla u|^p\,dx + \int_{\dOm} uf\,dP_\Om,
\end{equation*}
where $|\nabla u|$ will be replaced by the (minimal $p$-weak) upper gradient of
$u$ in the metric setting.

In~\cite{KK} issues of existence and stability of solutions to the general
Neumann boundary value problem for a class of $p$-Laplace-type operators were
considered. The paper~\cite{RS} gave a computational scheme for constructing
solutions to the general Neumann boundary value problem for the Laplacian in
three-dimensional Euclidean domains with piecewise smooth boundary. The
paper~\cite{CM} of Cianchi and Maz$'$ya studied regularity of solutions to the
Neumann boundary value problem, for Lipschitz domains, related to the
$p$-Laplace and more general operators, but the Neumann data they consider is
the natural boundary condition, i.e., constant zero data. The work of Agmon,
Douglis and Nirenberg~\cite{ADN1, ADN2}, M. Taylor~\cite{Tay},
Cranny~\cite{Cr}, and the recent work of Kenig, Lin, and Shen~\cite{KLS},
Milakis and Silvestre~\cite{MS}, studied regularity of solutions to the general
Neumann problem for homogeneous and rapidly oscillating elliptic PDEs for
$C^{1,\alpha}$-domains in Euclidean setting. The work of Dancer, Daners and
Hauer~\cite{DDH} explored the behavior of solutions to the zero (natural)
Neumann boundary value problem for the $p$-Laplacian on Euclidean domains whose
complement is a compact set, showing that solutions that have a certain decay
property at $\infty$ (decay to zero) have to vanish identically on the domain.

The study of the Neumann problem in non-smooth settings is currently sparse. In
the more general setting of Carnot groups, Nhieu~\cite{N} studied the existence
and uniqueness of solutions to the Neumann boundary value problem for the
sub-Laplacian operator (corresponding to $p=2$) on bounded Lipschitz domains.
More explicit computations, in terms of Green's functions, were given by Dubey,
Kumar, and Misra in~\cite{DKM}. Mixed boundary problems and homogenization for
domains in Heisenberg groups were considered by Tchou~\cite{Tch}, Biroli,
Tchou, and Zhikov~\cite{BTZ}. In the non-smooth metric setting, an analog of
the Dirichlet problem for the $p$\mbox{-}Laplacian was initiated in~\cite{KS}
and is currently an active area of research (see for example~\cite{BBbook}). In
this paper, we propose an analog of the Neumann boundary value problem adapted
to the non-smooth setting by using the tools of calculus of variation and ideas
due to De Giorgi, Giaquinta, and Giusti.

The challenge in our situation is three-fold. First, unlike~\cite{ADN1, ADN2,
Tay, KLS, N, DKM, Tch, BTZ}, our problem is non-linear even in the Euclidean
setting ($p$-Laplacian with $1<p<\infty$); second, lack of smoothness
structure, especially at the boundary, and so we have no notion of $C^1$-domain
in the metric setting; third, unlike in~\cite{N, DKM, Tch, BTZ}, lack of the
Euler--Lagrange (PDE) equation corresponding to the energy minimization
problem, since the upper gradient structure on the metric space might not come
from an inner product structure. We therefore do not have access to tools such
as the Euler--Lagrange equation nor layer potentials as used for example in the
work of Maz$'$ya and Poborchi~\cite{MP}. Thus the results obtained in this
paper are, not surprisingly, weaker than those of~\cite{ADN1, ADN2, Tay, KLS},
but on the other hand, they are applicable to a wider class of operators than
linear elliptic operators and are applicable to a wider range of domains even
in the Euclidean setting (such as Lipschitz domains that might perhaps not  be
$C^1$-domains).  We show that solutions exist (Theorem~\ref{thm:exists}) and
are bounded at the boundary of the domain (Theorem~\ref{thm:bounded}). As
mentioned above, the key step is to identify an analog of the De Giorgi
inequality adapted to the problem, see Theorem~\ref{thm:DeGiorgi}. Furthermore,
we apply this version of De Giorgi inequality to prove continuity of solutions
for certain values of $p$ at a.e.\@ boundary point as well as at every boundary
point in whose neighborhood the Neumann data does not change its sign
(Theorem~\ref{thm:Main2} and~Theorem~\ref{thm:Main3}).

The paper~\cite{GaMPR} by Garc\'ia-Azorero, Manfredi, Peral and Rossi studied
the Neumann boundary value problem for the $p$-Laplace operator in the
Euclidean setting and showed for smooth domains with continuous boundary data
that the the solutions for a given data are unique up to additive constants.
Their proof used the Euler--Lagrange formulation of the problem, an approach
that is not available to us in the non-smooth setting. We obtain a weaker
uniqueness property, namely that the minimal $p$-weak upper gradients of the
solutions are all equal, see Lemma~\ref{lem:almost-unique}. However, if the
metric measure space $X$ has a Cheeger-type differential structure (that is, a
first order Taylor theorem is satisfied for Lipschitz functions with respect to
a vector bundle on $X$) such that the minimal $p$-weak upper gradient of $u \in
N^{1,p}(X)$ is equal to the norm of its Cheeger derivative, then we can
conclude from Lemma~\ref{lem:almost-unique} that solution is, in fact, unique
(up to an additive constant), taking into consideration also that the set of
solutions form a convex set. Metric spaces endowed with such a differential
structure are said to be infinitesimally Hilbertian, see \cite{AGS} and
\cite{Gig}. Even in some weighted Euclidean setting, we do not have this
Hilbertian property, see~\cite{KSZ}. In infinitesimally Hilbertian spaces, one
has also access to the corresponding Euler--Lagrange equation, which enables
obtaining the uniqueness of solutions (up to an additive constant) via PDE
methods.

The rest of this paper is organized as follows. We give the needed definitions
in Section~2. In Section~3 we describe the Neumann problem in the metric
setting using the language of calculus of variation, and discuss the needed
tool of boundary trace of Sobolev functions. Existence of solutions for bounded
boundary data is studied in Section~4, with Theorem~\ref{thm:exists} declaring
the existence of solutions. A weak analog of uniqueness of solutions is given
in Lemma~\ref{lem:almost-unique} in this section as well. The focus of
Section~5 is to prove that the solutions are necessarily bounded, see
Theorem~\ref{thm:bounded}. The key inequality that is an analog of the De
Giorgi inequality is also given in this section, in Theorem~\ref{thm:DeGiorgi}.
In Section~6 we discuss regularity of solutions at the boundary, and show that
at boundary points where the boundary data is non-negative the solution must
necessarily be a subminimizer and hence is upper semicontinuous there. For
metric spaces with measure $\mu$ that have a strong regularity, known as
Ahlfors regularity, we show in the final section of this paper that the
solutions are continuous at boundary points where the boundary data does not
change sign (Theorem~\ref{thm:Main2}) and that the solutions are continuous at
Hausdorff co-dimension $1$-almost every boundary point
(Theorem~\ref{thm:Main3}).
\section{Preliminaries}
The triplet $(X,\dd,\mu)$ denotes a metric measure space. We say that $\mu$ is
\emph{doubling} if there is a constant $C_D$ such that for each $x\in X$ and
$r>0$,
\[
  0<\mu(B(z,2r))\le C_D\, \mu(B(z,r))<\infty.
\]
\begin{lemma}[{see e.g.\@ \cite[Lemma~3.3]{BBbook}}]
There is $s>0$ such that
\begin{equation}
\label{eq:mu-dimension}
  \frac{\mu(B(x,r))}{\mu(B(y,R))} \ge C \biggl(\frac{r}{R} \biggr)^s
\end{equation}
for all $0<r\le R$, $y\in X$, and $x\in B(y,R)$.
\end{lemma}
Note that we can always take $s$ to be as large as we wish. Therefore from now
onwards we assume that $s>1$. We also say that $\mu$ is \emph{Ahlfors
$s$-regular at scale $r_0>0$} if there is a constant $C>0$ such that whenever
$x\in X$ and $0<r<r_0$, we have
\[
\frac{1}{C}r^s\le \mu(B(x,r))\le C\, r^s.
\]

In what follows, the space $L^1_\loc(X)$ consists of functions on $X$ that are
integrable on \emph{bounded} subsets of $X$.

A Borel function $g: X\to[0,\infty]$ is an upper gradient of
$u:X\to\RR\cup\{\pm\infty\}$ if the following inequality holds for all
(rectifiable) curves $\gamma:[a,b]\to X$, (denoting  $x=\gamma(a)$ and
$y=\gamma(b)$),
\[
  |u(y)-u(x)|\leq\int_{\gamma}g\,ds
\]
whenever $u(x)$ and $u(y)$ are both finite, and $\int_{\gamma}g\,ds=\infty$
otherwise. The notion of upper gradients, first formulated in~\cite{HK} (with
the terminology ``very weak gradients''), plays the role of $|\nabla u|$ in the
metric setting where no natural distributional derivative structure exists.
\begin{definition}
\label{df:N1p}
The \emph{Newtonian space} $N^{1,p}(X)$ is defined by
\[
  N^{1,p}(X) = \Bigl\{u\in L^p(X): \|u\|_{N^{1,p}(X)} \coloneq \| u \|_{L^p(X)} +
   \inf_g \|g\|_{L^p(X)} <\infty \Bigr\},
\]
where the infimum is taken over all upper gradients $g$ of $u$.
\end{definition}
Let us point out that we assume that functions are defined everywhere, and not
just up to equivalence classes \mbox{$\mu$-}almost everywhere. This is
essential for the notion of upper gradients since they are defined by a
pointwise inequality.
\begin{definition}\label{def:rel-cap}
Given a ball $B=B(x,r)\subset X$ and a set $E\subset B$, the \emph{relative
$p$-capacity} of $E$ with respect to $2B=B(x,2r)$ is given by
\[
  \text{cap}_p(E,2B)\coloneq\inf_u \int_{2B}g_u^p\, d\mu,
\]
where the infimum is over all functions $u\in N^{1,p}(X)$ for which $u\ge 1$ on
$E$ and $u=0$ on $X\setminus 2B$.
\end{definition}
It follows from~\cite[Proposition~6.16]{BBbook} that
\begin{equation}\label{eq:rel-cap-Ball}
  \frac{\mu(E)}{C\, r^p}\le \text{cap}_p(E,2B)\le C\, \frac{\mu(B)}{r^p}.
\end{equation}
\begin{definition}[cf.\@ \cite{A1}]
A metric space $X$ supports a \emph{$p$-Poincar\'{e} inequality} with $p\in[1,
\infty)$ if there exist positive constants $\lambda$ and $C$ such that for all
balls $B\subset X$ and all $u\in L^1_\loc(X)$,
\begin{equation}
  \label{eq:pPI}
  \fint_{B}|u-u_{B}|\,d\mu\leq C\rad(B)\biggl(\fint_{\lambda B} g^p\,d\mu\biggr)^{1/p}\,.
\end{equation}
\end{definition}
Here and in the rest of the paper, $f_A$ denotes the \emph{integral mean} of a
function $f\in L^0(X)$ over a measurable set $A \subset X$ of finite positive
measure, defined as
\[
  f_A = \fint_A f\,d\mu = \frac{1}{\mu(A)} \int_A f\,d\mu
\]
whenever the integral on the right-hand side exists, not necessarily finite
though. Furthermore, given a ball $B=B(x,r) \subset X$ and $\lambda>0$, the
symbol $\lambda B$ denotes the inflated ball $B(x, \lambda r)$.

We next give an analog of the notion of sets of finite perimeter, as formulated
in~\cite{M}, see~\cite{EvaG92, AFP, Zie89} for the Euclidean setting.
\begin{definition}\label{def:finitePerimeter}
A Borel set $E\subset X$ is said to be of finite perimeter if there is a
sequence $(u_k)_{k\in\NN}$ from $N^{1,1}(X)$ such that $u_k\to\chi_E$ in
$L^1(X)$ and
\[
  \liminf_{k\to\infty}\int_Xg_{u_k}\, d\mu<\infty.
\]
The \emph{perimeter} $P_E(X)$ of $E$ is the infimum of the above limit infima
over all such sequences $(u_k)_k$ as above. Given an open set $U\subset X$, the
perimeter of $E$ in $U$ is
\[
  P_E(U)=\inf\biggl\{
   \liminf_{k\to\infty}\int_Ug_{u_k}\, d\mu \colon
  (u_k)_{k\in\NN}\subset N^{1,1}(U), u_k\to \chi_{E\cap U}\text{ in }L^1(U)\biggr\}.
\]
\end{definition}
An analogous notion (using Lipschitz functions rather than functions in
$N^{1,1}(X)$) was proposed in~\cite{M}, but the notion given there agrees with
ours when the measure on $X$ is doubling and $X$ supports a $1$-Poincar\'e
inequality. A direct translation of the proof given in~\cite{M} shows that the
Carath\'eodory extension of $P_E$ to subsets of $X$ is a finite Radon measure
on $X$. In~\cite{A1}, Ambrosio demonstrated that if the measure on $X$ is
doubling and supports a $1$-Poincar\'e inequality, then the Radon measure $P_E$
is equivalent to the co-dimension $1$ Hausdorff measure restricted to the
measure-theoretic boundary $\partial_mE$ of $E$. Here, $x\in \partial_mE$ if
and only if $x\in X$ and
\[
\limsup_{r\to0^+}\frac{\mu(B(x,r)\cap E)}{\mu(B(x,r))}>0\quad\text{and}\quad
\limsup_{r\to0^+}\frac{\mu(B(x,r)\setminus E)}{\mu(B(x,r))}>0.
\]

Given $A\subset X$, we define its \emph{co-dimension $1$ Hausdorff measure}
$\mathcal{H}(A)$ by
\begin{equation}\label{eq:deff-mathcal}
\mathcal{H}(A)=\lim_{\delta\to 0^+}\ \inf\biggl\{\sum_i\frac{\mu(B_i)}{\rad(B_i)}
\colon B_i\text{ balls in }X, \rad(B_i)<\delta, A\subset \bigcup_iB_i\biggr\}.
\end{equation}

Thus, the results of~\cite{A1} show that there is a constant $C\ge 1$ such that
whenever $E\subset X$ is of finite perimeter and $K\subset X$ is a Borel set,
we must have
\[
\frac{1}{C}\mathcal{H}(K\cap\partial_mE)\le P_E(K)\le C\mathcal{H}(K\cap\partial_mE).
\]
See~\cite{M, A1, AMP, AFP} for more on sets of finite perimeter and associated
functions of bounded variation in the metric setting. The paper~\cite{ADG}
studies connections between the relaxation of the co-dimension $1$ Minkowski
content of the boundary and the perimeter measure.
\section{Statement of the problem and standing assumptions}
In this paper, $1<p<\infty$ and $X$ is a complete metric space equipped with a
doubling measure $\mu$ supporting a $p$-Poincar\'e inequality.
\begin{definition}\label{defn:problem}
Let $\Om$ be a bounded domain (non-empty, connected open set) in $X$ with
$X\setminus \Om$ of positive measure such that $\Om$ is also of finite
perimeter with perimeter measure $P_\Om$. Let $f:\dOm\to\RR$ be a bounded
$P_\Om$-measurable function with $\int_{\dOm}f\, dP_\Om=0$. We say that a
function $u:\Om\to\RR$ is a \emph{$p$-harmonic solution to the Neumann boundary
value problem with boundary data} $f$ if $u\in N^{1,p}(\Om)$ and
\begin{equation}\label{eq:I-functional}
I(u)\coloneq \int_\Om g_u^p\, d\mu+\int_{\dOm}Tu\, f\, dP_\Om\le
\int_\Om g_v^p\, d\mu+\int_{\dOm}Tv\, f\, dP_\Om=I(v)
\end{equation}
for every $v\in N^{1,p}(\Om)$. Here $g_u$ and $g_v$ are the minimal $p$-weak
upper gradients of $u$ and $v$ in $\Om$, respectively, and $Tu$ and $Tv$ denote
the traces of $u$ and $v$ on $\dOm$, respectively.
\end{definition}
When considering the original Neumann boundary value problem~\eqref{eq:PDE}, we
see that adding a constant to a solution gives us another solution. Thus, the
Neumann boundary data $f$ has to satisfy the compatibility condition
\[
  \int_{\dOm} f\,dP_\Om = 0
\]
so that the value of the functional $I$ as defined in \eqref{eq:I-functional}
is invariant with respect to adding a constant to a solution.
\begin{definition}[Assumptions on $\Om$]\label{Assume-Om}
We will assume in this paper that here is a constant $C\ge 1$ such that for all
$x\in \dOm$, $z\in\Om$, and $0<r\le \diam(\Om)$, we have
\begin{equation}\label{density}
 \mu(B(z,r)\cap\Om)\ge C^{-1}\mu(B(z,r)),
\end{equation}
and
\begin{equation}\label{boundary-Ahlfors-regularity}
 C^{-1}\frac{\mu(B(x,r))}{r}\le P_\Om(B(x,r))\le C\frac{\mu(B(x,r))}{r}.
\end{equation}
We also assume that $(\Om, d|_\Om, \mu\lfloor_\Om)$ admits a $p$-Poincar\'e
inequality with dilation factor $\lambda=1$, where $p \in (1, \infty)$ is equal
to the exponent in \eqref{eq:I-functional}.
\end{definition}
Under the above assumptions, we also have a Sobolev-type inequality for
$\Om$,
\begin{equation}
  \label{eq:SobolevEmb-Om}
  \| u - u_\Om\|_{L^p(\Om)} \le C \| g_u \|_{L^p(\Om)},
\end{equation}
where $C=C(\Om, C_D, p, \ldots)$. This Sobolev-type embedding follows from
classical embedding results of~\cite{HaK}.

The property of satisfying~\eqref{boundary-Ahlfors-regularity} will be called
\emph{Ahlfors codimension $1$ regularity} of $P_\Om$.

The condition~\eqref{density} together with
condition~\eqref{boundary-Ahlfors-regularity} implies that $\mu(\dOm)=0$, and
that $\Om$ is of finite perimeter. It follows by the results of
Ambrosio~\cite{A1} that if $X$ supports a $1$-Poincar\'e inequality, then
$P_\Om\approx\mathcal{H}\vert_{\dOm}$; thus the above
condition~\eqref{boundary-Ahlfors-regularity} remains valid (with a different
constant $C$ perhaps) if $P_\Om$ is replaced with $\mathcal{H}$. Examples of
domains satisfying the above conditions include domains with quasiminimal
boundary surfaces as studied in~\cite{KinnKorLorSh}.

Domains that are sets of finite perimeter are the natural class of domains for
which the Neumann boundary value problem makes sense, as this is the largest
class of domains for which, at least in the Euclidean setting, a form of
Gauss-Green theorem holds true, see the work~\cite{CTZ} of Chen, Torres and
Ziemer (for metric space analogs see~\cite{MMS}).

The assumption that $\lambda=1$ in the $p$-Poincar\'e inequality supported by
$\Om$ is satisfied for example if $\Om$ is a geodesic domain, that is, for each
$x,y\in\Om$ there is a curve $\gamma\subset\Om$ with end points $x,y$ such that
the length of $\gamma$ is equal to $d(x,y)$. It then follows from the results
of~\cite{HaK} that the factor $\lambda$ in the $p$-Poincar\'e inequality can be
chosen to equal $1$ (perhaps at the expense of a larger constant $C$). The
assumption that $\lambda=1$ is a mere technicality here, assumed for the sake
of simplifying the computations; they get more complicated when $\lambda>1$,
but the results still remain true as an interested reader can verify.
\begin{definition}[Traces of Sobolev functions on $\dOm$]\label{defn:traces}
Under the standing assumptions on $\Om$ given above in
Definition~\ref{Assume-Om}, there is a bounded linear \emph{trace operator}
\[
T: N^{1,p}(\Om) \to L^{\widetilde{p}}(\dOm)
\]
for every $\widetilde{p} < p^*$, where $p^* = p(s-1)/(s-p)$ if $p<s$, and $p^*
= \infty$ if $p\ge s$. This trace operator is given as follows. For $u\in
N^{1,p}(\Om)$, $\mathcal{H}$-almost every $x\in \dOm$, there exists
$Tu(x)\in\RR$ such that
\[
  \lim_{r\to 0^+}\fint_{B(x,r)\cap\Om}|u-Tu(x)|\, d\mu = 0.
\]
Here, $s$ is the lower mass bound exponent from~\eqref{eq:mu-dimension}. If
$p>s$, then we can allow for $\widetilde{p}=\infty$ as well, though this is not
of importance to us in this paper.
\end{definition}
Existence of such a trace operator follows from~\cite[Theorem~3.4]{LS}. The
following trace theorem is a specific case of the trace theorem found
in~\cite{MalTr}, but for the convenience of the reader we provide its proof
here.
\begin{proposition}[{cf.\@ \cite{MalTr}}]
\label{pro:trace}
Assume that $\Om$ is a length space and that the dilation factor $\lambda = 1$
in the Poincar\'e inequality \eqref{eq:pPI}. Suppose that $p < s$. Let
$\widetilde{p} \in (p, p^*)$. Then, the trace operator $T: N^{1,p}(\Om) \to
L^{\widetilde{p}}(\dOm)$ is linear, bounded, and for every $0<\eps<(s-1)(p^* -
\widetilde{p})/(\widetilde{p}p^*)$ there is $C>0$ such that $T$ satisfies
\begin{align*}
  \| Tu \|_{L^{\widetilde{p}}(\dOm \cap B)} & \le C \biggl(r^{(s-1)(\frac{1}{\widetilde{p}}-\frac{1}{p^*})-\eps}\|g_u\|_{L^p(\Om \cap  B)}
  + \frac{P_\Om(\dOm \cap B)^{1/\widetilde{p}}}{\mu(\Om \cap B)} \|u\|_{L^1(\Om \cap B)}\biggr) \\
  & = C \biggl(r^{1-\frac{1}{\widetilde{p}} - \aleph}\|g_u\|_{L^p(\Om \cap  B)}
  + \frac{P_\Om(\dOm \cap B)^{1/\widetilde{p}}}{\mu(\Om \cap B)} \|u\|_{L^1(\Om \cap B)}\biggr)
\end{align*}
for every ball $B=B(z,r)$ with $z\in\dOm$, where $\aleph = s(\frac{1}{p} -
\frac{1}{\widetilde{p}})+\eps$. If $\mu$ is Ahlfors $s$-regular at scale
$r_0>0$, then the estimates above hold with $\eps=0$ whenever $r<r_0$.
\end{proposition}
\begin{remark}
The requirement that $\lambda=1$ is not restrictive, since length spaces
supporting a $p$-Poincar\'e inequality will support such an inequality with
$\lambda=1$ (perhaps at the expense of a larger constant $C$), see for
example~\cite{HaK, BBbook}.
\end{remark}
\begin{proof}
Let $u \in N^{1,p}(\Om)$ and fix a ball $B=B(z,r)$ with $z\in\dOm$. If $\mu$ is
$s$-regular and $\eps=0$, let $0<\widetilde{\eps}<(s-1)(\frac{1}{\widetilde{p}}
- \frac{1}{p^*})$ be arbitrary. Otherwise, let $\widetilde{\eps} = \eps$.

For every point $x \in B \cap \dOm$, define $r_x = \frac12(r - \dd(x,z)) \le
\frac12 \dist(\{x\}, \dOm \setminus B)$. Since $\Om$ is a length space, we can
find an arc-length parametrized curve $\gamma_x: [0, l_x] \to \overline\Om$
such that $\gamma_x(0) = z$, $\gamma_x(l_x) = x$, $\gamma_x\bigl( (0, l_x)
\bigr) \subset \Om$, and $l_x \le (1+\delta) \dd(x,z)$, where the constant
$\delta=\delta_x \in (0, 1)$ is chosen such that $(1+\delta)l_x < r$.

Next, we will construct a finite decreasing sequence of balls whose centers lie
on $\gamma_x$ and all the balls contain the point $x$ and are contained in
$B(z, r)$. Let $N = N_x = \lceil \log_2 (2r/r_x) \rceil$. For each $k=0, 1,
\ldots, N$, let $r_k = (\delta^{{k+1}} + 2^{-k}) l_x$ and let $x_k =
\gamma_x\bigl((1-2^{-k}) l_x\bigr)$. Then, we define $B_k = B(x_k, r_k)$.

It follows from the triangle inequality that $B_{k+1} \subset B_k \subset
B(z,r)$, and $x \in B_k$ for all $k=0,1,\ldots N$. For $k>N$, we define $B_k =
B(x, 2^{-k} l_x) \subset B(x, r_x)$. From Defintion~\ref{defn:traces}, we see
that $Tu(x) = \lim_{k\to \infty} \fint_{B_k \cap \Om}u\,d\mu$ for
$P_\Om$-a.e~$x\in B \cap\dOm$. We can thus estimate the difference $|u_{B\cap
\Om} - Tu(x)|$ using the chain of balls $\{B_k\}_{k=0}^\infty$. For the sake of
brevity, let $\zeta = (s-1)(\frac{1}{\widetilde p} - \frac{1}{p^*}) -
\widetilde{\eps}$, which can be simplified since
\[
  \zeta = (s-1)\biggl(\frac{1}{\widetilde p} - \frac{s-p}{p (s-1)}\biggr) - \widetilde{\eps}
   = \frac{s-1}{\widetilde p} - \frac{s}{p}+1 - \widetilde{\eps} = 1- \frac{1}{\widetilde{p}} - \widetilde\aleph <1,
\]
where $\widetilde\aleph = s(1/p - 1/\widetilde p) + \widetilde{\eps} > 0$.
Then, the doubling condition and the $p$-Poincar\'e inequality yield that
\begin{align*}
  |&u_{B\cap \Om} - Tu(x)| \le |u_{B\cap\Om} - u_{{B_0}\cap\Om}| + \sum_{k=1}^\infty |u_{B_k\cap\Om} - u_{{B_{k-1}}\cap\Om}| \\
  & \le C \biggl[ \fint_{B\cap\Om} |u-u_{B\cap\Om}|\,d\mu + \sum_{k=0}^\infty \fint_{B_k\cap\Om} |u - u_{B_k\cap\Om}|\,d\mu \biggr]\\
  & \le C \biggl[ r \biggl(\fint_{ B\cap\Om} g^p \,d\mu\biggr)^{1/p}
     + \sum_{k=0}^\infty 2^{-k}r \biggl(\fint_{B_k\cap\Om} g^p \,d\mu\biggr)^{1/p} \biggr]\\
  & \le C \biggl[r^\zeta \biggl(r^{p-p\zeta} \fint_{B\cap\Om} g^p \,d\mu\biggr)^{1/p}
     + \sum_{k=0}^\infty (2^{-k}r)^\zeta \biggl((2^{-k}r)^{p-p\zeta}\fint_{B_k\cap\Om} g^p \,d\mu\biggr)^{1/p} \biggr]\\
   & \le C r^{\zeta} M^*_{p-p\zeta, p} g(x),
\end{align*}
where $M^*_{\varpi, p}$ denotes a restricted non-centered fractional maximal
operator, defined for $f \in L^p( B \cap \Om)$ by
\[
  M_{\varpi, p}^* f(x) = \sup_{\substack{x \ni B_0\\ \text{ball } B_0 \subset  B}}
     \biggl(\rad(B_0)^\varpi \fint_{B_0\cap \Om} |f|^p\,d\mu\biggr)^{1/p}, \quad x\in B \cap \dOm,
\]
where $\varpi\coloneq p-p\zeta = s - \frac{p}{\widetilde{p}}(s-1) +
\widetilde{\eps} p> 1$ as $p<\widetilde{p}$. Boundedness of the fractional
maximal operator for $\varpi>1$ can be proven via the standard $5$-covering
lemma similarly as in~\cite[Lemma~6.3]{GKS}, whose proof however needs to be
modified because of the possible lack of Ahlfors $s$-regularity. In order to
make the proof in~\cite{GKS} work (with some straightforward modifications),
one needs the following non-trivial key estimate for an arbitrary ball $D$
centered in $\overline\Om$ with $D\cap \dOm \neq \emptyset$:
\[
  P_\Om(5D \cap \dOm)  \le C  \frac{\mu(D)}{\rad(D)} \le C\biggl(\frac{\mu(D)}{\rad(D)^{\varpi}}\biggr)^{(s-1)/(s-\varpi)},
\]
where $(s-1) / (s-\varpi) > 1$. The latter inequality is equivalent to
\[
  \frac{\mu(D)^{\frac{s-1}{s-\varpi}-1}}{\rad(D)^{\varpi\frac{s-1}{s-\varpi}-1}}
  = \biggl(\frac{\mu(D)}{\rad(D)^s}\biggr)^{(\varpi-1)/(s-\varpi)}
  \ge C \biggl(\frac{\mu(\Om)}{\diam(\Om)^s}\biggr)^{(\varpi-1)/(s-\varpi)} = C,
\]
which can be obtained from \eqref{eq:mu-dimension}. Thus, $M_{\varpi, p}^*:
L^p( B \cap \Om) \to \text{weak-}L^{p_\varpi}(B \cap \dOm)$ is bounded, where
$p_\varpi = p \frac{s-1}{s-\varpi} = \frac{s-1}{((s-1)/\widetilde{p}) -
\widetilde{\eps}} > \widetilde{p}$. Then,
\begin{align}
\notag
   \| u_{B\cap \Om} - Tu\|_{L^{\widetilde{p}}(B\cap \dOm)} & \le C r^{\zeta} \|M^*_{\varpi, p}g\|_{L^{\widetilde{p}}(B\cap \dOm)} \\
\notag
   & \le C r^{\zeta} P_\Om(B\cap \dOm)^{1/\widetilde{p}-1/p_\varpi}  \|M^*_{\varpi, p}g\|_{\text{weak-}L^{p_\varpi}(B\cap \dOm)} \\
\notag
  & \le C r^{\zeta} P_\Om(B\cap \dOm)^{1/\widetilde{p}-1/p_\varpi} \|g\|_{L^{p}( B\cap \Om)}\\
\label{eq:trace-proofend}
  & \le C r^{(s-1)({1}/{\widetilde{p}}-{1}/{p^*})-\widetilde\eps} P_\Om(B\cap \dOm)^{\widetilde\eps/(s-1)} \|g\|_{L^{p}( B\cap \Om)}\\
\notag
  & \le C r^{(s-1)({1}/{\widetilde{p}}-{1}/{p^*})-\widetilde\eps} P_\Om(\dOm)^{\widetilde\eps/(s-1)} \|g\|_{L^{p}( B\cap \Om)}\\
\notag
  & \le C r^{1-1/\widetilde{p}-\widetilde\aleph} \|g\|_{L^{p}( B\cap \Om)}\,,
\end{align}
where $C$ depends among others on $\widetilde{p}$, $p^*$, $\widetilde\aleph$
(and hence on $\widetilde{\eps}>0$), and $P_\Om(\dOm)$. Finally, the triangle
inequality yields that
\begin{align*}
\|Tu\|_{L^{\widetilde{p}}(B\cap \dOm)} & \le  C r^{1-1/\widetilde{p}-\widetilde\aleph} \|g\|_{L^{p}( B\cap \Om)}
+ \| u_{B\cap \Om}\|_{L^{\widetilde{p}}(B\cap \dOm)} \\
 &\le C r^{1-1/\widetilde{p}-\widetilde\aleph} \|g\|_{L^{p}( B\cap \Om)}
 + P_\Om(B \cap \dOm)^{1/\widetilde{p}} \frac{\| u\|_{L^1(B\cap \dOm)}}{\mu(B\cap \dOm)}\,.
\end{align*}
Recall that we have chosen $\widetilde\eps = \eps$ whenever $\eps > 0$, and
hence $\widetilde\aleph = \aleph$.

Suppose now that $\eps = 0 < \widetilde{\eps}$ when $\mu$ is $s$-regular at
scale $r_0$. Then, $P_\Om$ is $(s-1)$-regular at scale $r_0$ in view
of~\eqref{boundary-Ahlfors-regularity}. If $r<r_0$, then $P_\Om(B\cap
\dOm)^{\widetilde\eps/(s-1)} \le C r^{\widetilde{\eps}}$
in~\eqref{eq:trace-proofend} above, which yields
\[
  \| u_{B\cap \Om} - Tu\|_{L^{\widetilde{p}}(B\cap \dOm)} \le C r^{(s-1)(\frac{1}{\widetilde p} - \frac{1}{p^*})} \|g\|_{L^p( B\cap \Om)}
    = C r^{1-{1}/{\widetilde p} - \aleph} \|g\|_{L^p( B\cap \Om)}.
\]
The rest of the computation is analogous as before. Here,
$\aleph=\widetilde{\aleph}-\widetilde{\eps}$.
\end{proof}
From now on, for ease of notation, the trace $Tu$ of $u$ will also be denoted
by $u$.

Throughout the paper $C$ represents various constants that depend solely on the
doubling constant, constants related to the Poincar\'e inequality, and the
constants related to~\eqref{density} and~\eqref{boundary-Ahlfors-regularity}.
The precise value of $C$ is not of interest to us at this time, and its value
may differ in each occurrence. Given expressions $a$ and $b$, we say that
$a\approx b$ if there is a constant $C\ge 1$ such that $C^{-1}a\le b\le C a$.
\section{Existence of a minimizer}
The natural space to look for a minimizer of $I$ would be $W^{1,p}(\Om)$ if we
worked in the Euclidean setting. In the metric setting, we will make use of the
Newtonian space $N^{1,p}(\Om)$ as a suitable counterpart of the Sobolev space.

Since we aim to obtain a unique representative of a solution and adding a
constant to a solution yields another solution, we will make use of the
following normalization
\[
  N^{1,p}_*(\Om) = \biggl\{u \in N^{1,p}(\Om): \int_\Om u\,dx = 0 \biggr\}.
\]
Observe that $u\equiv 0$ is a candidate for the infimum in the definition of
$I(u)$, \eqref{eq:I-functional}. Therefore,
\[
  \inf_{u \in N^{1,p}_*(\Om)} I(u) \le 0.
\]
To show existence of a minimizer, we need to prove that $I(u)$ is bounded below
for $u \in N^{1,p}_*(\Om)$ and that the functional is sequentially lower
semi-continuous.
Based on the relation between $p\in(1, \infty)$ and the ``upper measure
dimension'' $s$ given by \eqref{eq:mu-dimension}, we will a priori distinguish
two possible integrability conditions of the Neumann boundary data.
\begin{proposition}\label{prop:I(u)LowerBound}
  Let $u \in N^{1,p}_*(\Om)$ and $f \in L^q(\dOm)$, where $q=1$ if $p>s$, and
  $\tfrac{p (s-1)}{s (p-1)}<q\le \infty$ if $p \le s$.  Then,
  \[
    I(u) \ge \|g_u\|^p_{L^p(\Om)} - C\|g_u\|_{L^p(\Om)} \|f\|_{L^q(\dOm)}\,.
  \]
\end{proposition}
\begin{proof}
The H\"older inequality yields that
  \[
  I(u) \ge \int_{\Om} |g_u|^p \,d\mu - \int_{\dOm} |uf|\,dP_\Om
  \ge \|g_u\|_{L^p(\Om)}^p - \|u\|_{L^{q'}(\dOm)} \|f\|_{L^q(\dOm)}\,.
  \]
It follows from the (proof of the) trace theorem for $N^{1,p}$ functions in
$p$-Poincar\'e spaces~\cite[Proposition~3.20]{MalTr} that $\|u\|_{L^{q'}(\dOm)}
\le C \|g_u\|_{L^p(\Om)}$ provided that $\int_\Om u\,d\mu = 0$.  Thus,
  \[
  I(u) \ge \|g_u\|_{L^p(\Om)} \bigl(\|g_u\|_{L^p(\Om)}^{p-1} - C \|f\|_{L^q(\dOm)}\bigr)\,.
  \qedhere
  \]
\end{proof}
Note that functions that are bounded and $P_\Om$-measurable on $\dOm$ are
automatically in $L^q(\dOm)$.
\begin{corollary}
\label{cor:IlowerBound}
There is a constant $C>0$, depending on $p$, $q$, and on the norm of the trace
operator $T: N^{1,p}(\Om) \to L^{q'}(\dOm)$ such that
\[
  I(u) \ge -C \|f\|_{L^q(\dOm)}^{p'}
\]
for every $u \in N^{1,p}_*(\Om)$.
\end{corollary}
\begin{proof}
The estimate can be shown by finding the absolute minimum of the function $t
\mapsto t^p - Ct \|f\|_{L^q(\dOm)}$, where $t\ge 0$.
\end{proof}
\begin{theorem}\label{thm:exists}
There is $u \in N^{1,p}_*(\Om)$ such that $I=I(u)$.
\end{theorem}
\begin{proof}
Let $I = \inf_{u\in N^{1,p}_*(\Om)} I(u)$ and let $\{u_k\}_{k=1}^\infty \subset
N^{1,p}_*(\Om)$ be a minimizing sequence, i.e., $I = \lim_{k\to\infty} I(u_k)$.
Let $g_k$ denote the $p$-weak minimal upper gradients of $u_k$, $k=1,2,\ldots$.
Using Proposition~\ref{prop:I(u)LowerBound}, we see that $I(v) \le 0$ requires
that $\|g_v\|_{L^p(\Om)} \le C^{1/(p-1)}\Vert f\Vert_{L^q(\dOm)}^{1/(p-1)}$.
Hence, the sequence $\{g_k\}_{k=1}^\infty$ is bounded in $L^p(\Om)$. Using
\eqref{eq:SobolevEmb-Om}, we obtain that $\{u_k\}_{k=1}^\infty$ is also bounded
in $L^p(\Om)$ since $(u_k)_\Om = 0$ by definition of $N^{1,p}_*(\Om)$. The
reflexivity of $L^p(\Om)$ yields that there are subsequences (which will also
be denoted by $\{u_k\}_{k=1}^\infty$ and $\{g_k\}_{k=1}^\infty$) and $u,g \in
L^p(\Om)$ such that $u_k \rightharpoonup u$ and $g_k \rightharpoonup g$ as
$k\to \infty$.

By Mazur's lemma, there are convex combinations
\[
  \widetilde{u}_k = \sum_{i=k}^{N(k)} \alpha_{k,i} u_i\quad\text{and}\quad
  \widetilde{g}_k = \sum_{i=k}^{N(k)} \alpha_{k,i} g_i,\quad k=1,2,\ldots,
\]
such that $\widetilde{u}_k \to u$ and $\widetilde{g}_k \to g$ in $L^p(\Om)$.
Observe that $\widetilde{g}_k$ are $p$-weak upper gradients of
$\widetilde{u}_k$ (not necessarily minimal, though). By
\cite[Proposition~2.3]{BBbook}, we can modify $u$ on a set of measure zero to
obtain a good representative such that $g$ is its $p$-weak upper gradient. In
what follows, we will consider $u$ to be such a good representative and hence
$u\in N^{1,p}(\Om)$. Applying \cite[Proposition~2.3 and Corollary~6.3]{BBbook}
and passing to a subsequence if necessary, we obtain that
\[
  \int_\Om g_u^p\,d\mu \le \liminf_{k\to\infty} \int_\Om g_{\widetilde{u}_k}^p\,d\mu,
\]
where $g_u$ and $g_{\widetilde{u}_k}$ are the minimal $p$-weak upper gradients
of $u$ and $\widetilde{u}_k$, respectively.

Since $\int_\Om u_k = 0$ for every $k=1,2,\ldots$ and $u_k \rightharpoonup u$,
we have that $\int_\Om u = 0$. Hence, $u\in N^{1,p}_*(\Om)$.

Considering that the trace operator $T: N^{1,p}(\Om) \to L^{q'}(\dOm)$ is
linear and the energy functional $v \mapsto \int_{\Om} g_v^p\,d\mu$ is convex,
we see that
\[
  I \le I(\widetilde{u}_k) = I\biggl(\sum_{i=k}^{N(k)} \alpha_{k,i} u_i\biggr)
  \le \sum_{i=k}^{N(k)} \alpha_{k,i} I(u_i) \to I\quad\text{as }k\to\infty.
\]
The continuity of the trace operator yields that
\begin{align*}
  I & \le I(u) = \int_\Om g_u^p\,d\mu + \int_{\dOm} uf\,dP_\Om \\
  & \le \liminf_{k\to \infty} \biggl(\int_\Om g_{\widetilde{u}_k}^p\,d\mu +
  \int_{\dOm} \widetilde{u}_k f\,dP_\Om \biggr)=\liminf_{k\to \infty} I(\widetilde{u}_k) = I.
\qedhere
\end{align*}
\end{proof}
\begin{lemma}
The set $M_I = \{u \in N^{1,p}_*(\Om) : I(u) = I\}$ of minimizers of $I(\cdot)$
is norm-closed and convex.
\end{lemma}
\begin{proof}
Let $\lambda \in (0,1)$ and let $u,v \in M_I$, then $w=\lambda u +
(1-\lambda)v$ satisfies
\[
  I(w) = I(\lambda u + (1-\lambda) v) \le \lambda I(u) + (1-\lambda) I(v) = I
\]
due to convexity of the functional $I(\cdot)$. Therefore, $w\in M_I$.

The set $M_I$ is closed due to sequential lower semi-continuity of $I(\cdot)$.
\end{proof}
\begin{lemma}\label{lem:almost-unique}
Suppose that $u,v \in M_I$. Then $\int_{\dOm} uf\,dP_\Om = \int_{\dOm}
vf\,dP_\Om$ and $g_u = g_v$ a.e.\@ in $\Om$.  Furthermore, if $u, v\in M_I$
then the functions $w_+, w_-$ given by
\[
w_+\coloneq\max\{u,v\}-\fint_\Om\max\{u,v\}\, d\mu
\]
and
\[
w_-\coloneq\min\{u,v\}-\fint_\Om\min\{u,v\}\, d\mu
\]
also belong to $M_I$.
\end{lemma}
\begin{proof}
For any $u$ and $v$ as in the hypothesis, set $w=\tfrac{u+v}{2}$. Then $g_w\le
\tfrac12[g_u+g_v]$.

By the uniform convexity of $t\mapsto t^p$ on $[0,\infty)$, we know that for
each $\delta>0$ there exists a positive constant $\eps =
\delta^p(2^{-1}-2^{-p})$ such that
\[
\left(\frac{a+b}{2}\right)^p\le \frac{a^p+b^p}{2}-\eps
\]
whenever $a,b\in[0,\infty)$ with $|a-b|\ge\delta$.

Suppose that $\{x\in\Om\, :\, g_v(x)\ne g_u(x)\}$ has positive measure. Then
there is some $\delta>0$ such that the measure of the set
\[
A_\delta\coloneq \{x\in\Om\, :\, |g_v(x)-g_u(x)|>\delta\}
\]
is positive. Then
\begin{align*}
I(u)=I(v)\le I(w)&\le \int_\Om \left(\frac{g_u+g_v}{2}\right)^p\, d\mu+\int_{\dOm}wf\, dP_\Om\\
  &\le \int_{A_\delta}\left[\frac{g_u^p+g_v^p}{2}-\eps\right]d\mu
	  +\int_{\Om\setminus A_\delta}\frac{g_u^p+g_v^p}{2}\, d\mu
	+ \int_{\dOm}wf\, dP_\Om\\
	&=\int_\Om \frac{g_u^p+g_v^p}{2}\, d\mu+\int_{\dOm}\frac{u+v}{2}f\, dP_\Om-\eps\, \mu(A_\delta)\\
	&\le I(u)-\eps\, \mu(A_\delta),
\end{align*}
which is not possible. Therefore $g_u=g_v$ $\mu$-a.e. in $\Om$, and hence it
also follows from $I(u)=I(v)$ that $\int_{\dOm}uf\, dP_\Om=\int_{\dOm}vf\,
dP_\Om$.

To prove the last part of the lemma, it suffices to show that
$w_+^0=\max\{u,v\}$ and $w_-^0=\min\{u,v\}$ are minimizers of the functional
$I$ corresponding to $f$. Note that $g_{w_-^0}\le
g_u\chi_{\{u<v\}}+g_v\chi_{\{u\ge v\}}=g_u$ and similarly $g_{w_+^0}\le g_u$.
Therefore
\[
I(w_{\pm}^0)\le \int_\Om g_u^p\, d\mu+\int_{\dOm}w_{\pm}^0\, f\, dP_\Om.
\]
Note that
\[
\int_{\dOm} [w_+^0+w_-^0]f\, dP_\Om=\int_{\dOm}[u+v]f\, dP_\Om
 =2\int_{\dOm}uf\, dP_\Om.
\]
It follows that if $\int_{\dOm} w_+^0 f\, dP_\Om>\int_{\dOm} u f\, dP_\Om$,
then $\int_{\dOm}w_-^0f\, dP_\Om<\int_ {\dOm}u f\, dP_\Om$, which would violate
the minimality of $I(u)$. Therefore we must have $\int_{\dOm} w_+^0 f\,
dP_\Om\le\int_{\dOm} u f\, dP_\Om$ and similarly, $\int_{\dOm} w_-^0 f\,
dP_\Om\le\int_{\dOm} u f\, dP_\Om$, which in turn implies that $I(w_{\pm}^0)\le
I(u)$, as desired.
\end{proof}
Observe that in infinitesimally Hilbertian spaces, the above uniqueness of the
minimal $p$-weak upper gradient together with convexity of the set $M_I$ imply
that the solution of the Neumann problem is in fact unique (up to an additive
constant).
\section{Boundedness of solutions, at the boundary}
We will use the De Giorgi method to prove that the minimizers are bounded near
the boundary of $\Om$. Local boundedness inside $\Om$ follows from previously
known results on $p$-energy minimizers in the metric setting \cite{KS}.

Let $u \in N^{1,p}(\Om)$ be a minimizer of
\begin{equation}
\label{eq:DeGiorgiFcnal}
  I(u) = \int_\Om g_u^p\,d\mu + \int_{\dOm} f u\,dP_\Om,
\end{equation}
where $f \in L^\infty(\dOm)$ is a Borel function.  The main goal of this
section is to prove that solutions are bounded whenever the boundary data $f$
is bounded.
\begin{theorem}\label{thm:bounded}
Let $\Om$ be a bounded domain in $X$ satisfying the assumptions given in
Definition~\ref{Assume-Om}, and let $f$ and $u$ be as above. Fix $R_0 \in (0,
\diam \Om)$. Then for each $x\in\dOm$ and $0<R<R_0/4$ we have that $|u|\le C_R$
on $\overline{\Om}\cap B(x,R)$, where $C_R$ depends on the doubling and
Poincar\'e inequality constants, $p$, $R$, $\|u\|_{L^1(B(x,R)\cap\dOm)}$,
$\|u\|_{L^p(B(x,R)\cap\Om)}$, and on $\Vert f\Vert_{L^\infty(\dOm\cap
B(x,2R))}$ alone.
\end{theorem}
To prove the above theorem we make use of the technique developed by De
Giorgi~\cite{Giu2}. To do so we first derive a De Giorgi type inequality
associated with the Neumann type problem considered here.
\begin{theorem}\label{thm:DeGiorgi}
There is a constant $C\ge 1$ such that given a minimizer $u$ as above on the
bounded domain $\Om\subset X$, $x\in\dOm$, $0<r<R\le R_0<\diam(\Om)/10$, and
$k\in\RR$, we have
\begin{align}
  \label{eq:g(u-k)+est}
  \int_{\Om\cap B(x,r)} g_{(u-k)_+}^p \,d\mu & \le \frac{C}{(R-r)^p} \int_{\Om\cap B(x,R)} (u-k)_+^p\,d\mu  \\
  & \quad + C \int_{\dOm \cap B(x,R)} |f|\cdot(u-k)_+\,dP_\Om.
  \notag
\end{align}
The constant $C$ depends solely on the doubling constant of $\mu$, the
Poincar\'e inequality constants, and $p$.
\end{theorem}
\begin{proof}
Let $x, r, R$ be as in the statement of the theorem, and let
\begin{equation}\label{eq:cut-off}
\eta_{r,R}(y)=\eta(y) = \bigl(1-\dist(y, B(x,r))/(R-r)\bigr)_+
\end{equation}
be a Lipschitz cut-off function. For $k\in\RR$ and $\rho>0$, define
\[
A(k,\rho) = \{y \in B(x,\rho) \cap \Om\, :\, u(y)>k\}\cup\{y\in B(x,\rho)\cap\dOm\, :\, T(u)(y)>k\}.
\]
Note that by our standing assumptions on $\dOm$, we automatically have
$\mu(\dOm)=0$, and so integrating over $A(k,\rho)\cap\Om$ with respect to $\mu$
is the same as integrating over $A(k,\rho)$ with respect to $\mu$. For the
function
\[
  v = u - \eta\cdot(u-k)_+ =
  \begin{cases}
    (1-\eta)(u-k)+k & \text{in } A(k,R), \\
    u & \text{otherwise},
  \end{cases}
\]
by the properties of upper gradient (see~\cite{BBbook}) such as the Leibniz
rule, we have
\begin{equation}
  g_v \le
  \begin{cases}
      (1-\eta)g_u + \frac{u-k}{R-r} \chi_{B(x,R)\setminus B(x, r)} & \text{in } A(k,R), \\
    g_u & \text{otherwise}.
  \end{cases}
\label{eq:gv_estimate}
\end{equation}
Since $v$ is a candidate for the minimizer of $I$, we have $I(u) \le I(v)$.
Thus,
\[
  \int_{\Om\cap B(x,R)} g_u^p\,d\mu + \int_{\dOm \cap B(x,R)} fu\,dP_\Om \le \int_{\Om\cap B(x,R)} g_v^p\,d\mu
    + \int_{\dOm \cap B(x,R)} fv\,dP_\Om.
\]
Subtracting $\int_{\Om\cap B(x,R)\setminus A(k,R)} g_u^p\,d\mu + \int_{\dOm
\cap B(x,R)} fu\,dP_\Om$ from both sides of the inequality yields that
\begin{equation}
  \label{eq:ARk-gugv-est}
  \int_{A(k,R)} g_u^p\,d\mu \le \int_{A(k,R)} g_v^p\,d\mu - \int_{\dOm \cap A(k,R)} f\eta\cdot(u-k)\,dP_\Om.
\end{equation}
From \eqref{eq:gv_estimate}, we obtain the almost everywhere pointwise estimate
\[
  g_v^p \le 2^p \biggl( g_u^p(1-\chi_{A(k,r)}) + \frac{(u-k)^p}{(R-r)^p}\biggr) \quad\text{on $A(k,R)$.}
\]
Plugging in this estimate into \eqref{eq:ARk-gugv-est} and making the
integration domain on the left-hand side smaller, we have
\begin{align*}
  \int_{A(k,r)} g_u^p\,d\mu & \le 2^p \int_{A(k,R)\setminus A(k,r)} g_u^p\,d\mu \\
  &\quad + \frac{2^p}{(R-r)^p} \int_{A(k,R)} (u-k)^p\,d\mu - \int_{\dOm \cap A(k,R)} f\eta\cdot(u-k)\,dP_\Om.
\end{align*}
Adding $2^p \int_{A(k,r)} g_u^p\,d\mu$, and then dividing by $(1+2^p)$ leads to
\begin{align}
\notag
  \int_{A(k,r)} g_u^p&\,d\mu  \le \theta \int_{A(k,R)} g_u^p\,d\mu \\
  & + \frac{\theta}{(R-r)^p} \int_{A(k,R)} (u-k)^p\,d\mu
    - \frac{1}{C} \int_{\dOm \cap A(k,R)} f\eta\cdot(u-k)\,dP_\Om,
  \label{eq:ArkARk-Cest}
\end{align}
where $\theta = 2^p/(1+2^p) \in (0,1)$ and $C=1+2^p\ge 1$.

Now, we can apply \cite[Lemma 6.1]{Giu2} with \eqref{eq:ArkARk-Cest} as the
starting inequality to obtain
\[
  \int_{A(k,r)} g_u^p \,d\mu \le \frac{C}{(R-r)^p} \int_{A(k,R)} (u-k)^p\,d\mu
  + C \int_{\dOm \cap A(k,R)} |f|\cdot(u-k)\,dP_\Om,
\]
This verifies~\eqref{eq:g(u-k)+est} and completes the proof of the theorem.
\end{proof}
\begin{remark}
If $f>0$ on $B(x,R_0)$, then the inequality \eqref{eq:ArkARk-Cest} can be made
simpler by omitting the last term, viz.,
\[
  \int_{A(k,r)} g_u^p \,d\mu  \le \theta \int_{A(k,R)} g_u^p\,d\mu
   + \frac{1}{(R-r)^p} \int_{A(k,R)} (u-k)^p\,d\mu.
\]
In such a case \cite[Lemma 6.1]{Giu2} provides us with an estimate
\[
  \int_{A(k,r)} g_u^p \,d\mu \le \frac{C}{(R-r)^p} \int_{A(k,R)} (u-k)^p\,d\mu,
\]
which holds for every $0<r<R<R_0$.
\end{remark}
\begin{lemma}\label{lem:decay-est}
Let $x \in \dOm$ and $0<r<R<R_0$ as above, and let
 $C_f = \|f\|^{1/p}_{L^\infty(\dOm\cap B(x,R_0))}$,
\[
u(k, r) =\biggl(\fint_{\Om\cap B(x,r)} (u-k)_+^p\,d\mu\biggr)^{1/p},
\]
and
\[
\psi(k, R) =\fint_{\dOm\cap B(x,R)} (u-k)_+\,dP_\Om.
\]
If $N_\loc^{1,p}(\Om)\subset L_\loc^{\kappa p}(\Om)$ and the trace operator $T:
N^{1,p}(\Om) \to L^{\widetilde\kappa p}(\dOm)$ is bounded for some $\kappa,
\widetilde{\kappa}>1$ and $0<\aleph<1$, then for all real numbers $h,k$ with
$h<k$, all positive $R, r$ with $R/2 \le r< R<R_0$, setting $\alpha \coloneq
1-\frac{1}{\kappa}$, and $\beta \coloneq 1-\frac{1}{\widetilde\kappa p}$ yields
that
\begin{align}\notag
  u(k,r) & \le C \biggl( \frac{u(h,R)}{k-h} \biggr)^{\alpha}\biggl(\frac{R}{R-r} u(h,R)
       + C_f R^{1-1/p} \psi(h,R)^{1/p}\biggr), \quad\text{and} \\
\label{eq:ukr-uhR,psikr-psihR}
  \psi(k,r) & \le C \biggl( \frac{\psi(h,R)}{k-h} \biggr)^{\beta}\biggl(\frac{R^{1-\aleph}}{R-r} u(h,R)
        +C_f R^{1-1/p -\aleph} \psi(h,R)^{1/p}\biggr).
\end{align}
If in addition $\mu$ is Ahlfors $s$-regular at scale $r_0>0$, then we also have
\begin{equation}\label{eq:Ahlfors-strong}
\psi(k,r)\le C  \biggl( \frac{\psi(h,R)}{k-h} \biggr)^{\beta}
   \biggl[ \frac{R}{R-r} u(h,R) + C_f R^{1-1/p}\psi(h,R)^{1/p} \biggr].
\end{equation}
\end{lemma}
We can always chose such $\kappa, \widetilde{\kappa}$, for instance, by
choosing $1<\kappa<s/(s-p)$ and $1<\widetilde{\kappa}<(s-1)/(s-p)$ as in
Proposition~\ref{pro:trace}. If $p$ is close to $s$ then $\kappa$ and
$\widetilde{\kappa}$ can be chosen to be as large as we like.
\begin{proof}
Due to self-improvement of $(1,p)$-Poincar\'e inequality, there is $\kappa>1$
such that $\Om$ supports a $(\kappa p, p)$-Poincar\'e inequality, see for
example~\cite{HaK, BBbook}. Here any choice of $1<\kappa\le s/(s-p)$ works,
where $s$ is the upper mass bound exponent of the doubling measure $\mu$ as
in~\eqref{eq:mu-dimension}.

Let $\widetilde{\eta}$ be the cut-off function $\eta_{r, (r+R)/2}$ as
in~\eqref{eq:cut-off}. Then, the H\"older inequality and the $(\kappa p,
p)$-Poincar\'e inequality for functions in $N^{1,p}(X)$ vanishing on
$X\setminus B(x,(r+R)/2)$ yield
\begin{align*}
  \fint_{\Om\cap B(x,r)} (&u-k)_+^p\,d\mu \le
  \biggl(\frac{\mu(A(k,r))}{\mu(B(x,r))} \biggr)^{1-1/\kappa} \biggl(\fint_{\Om\cap B(x,r)} (u-k)_+^{\kappa p}\,d\mu\biggr)^{1/\kappa} \\
  & \le C \biggl(\frac{\mu(A(k,r))}{\mu(B(x,r))} \biggr)^{1-1/\kappa}
      \biggl(\fint_{\Om\cap B(x,(r+R)/2)} \bigl(\widetilde{\eta}(u-k)_+\bigr)^{\kappa p}\,d\mu\biggr)^{1/\kappa} \\
  & \le C \biggl(\frac{\mu(A(k,r))}{\mu(B(x,r))} \biggr)^{1-1/\kappa} R^p \fint_{\Om\cap B(x,(r+R)/2)} g_{\widetilde{\eta}(u-k)_+}^p\,d\mu \\
  & \le C \biggl(\frac{\mu(A(k,r))}{\mu(B(x,r))} \biggr)^{1-1/\kappa} R^p \fint_{\Om\cap B(x,(r+R)/2)} g_{(u-k)_+}^p + \frac{(u-k)_+^p}{(R-r)^p}\,d\mu,
\end{align*}
where the product rule (Leibniz rule) for ($p$-weak) upper gradients was used
in the last step. Estimating the integral of $g_{(u-k)_+}^p$ via
\eqref{eq:g(u-k)+est} gives
\begin{align*}
  \fint_{\Om\cap B(x,r)} (u-k)_+^p\,d\mu \le C \biggl(\frac{\mu(A(k,r))}{\mu(B(x,r))}\biggr)^{1-1/\kappa} \biggr[\frac{R^p}{(R-r)^p} \fint_{\Om\cap B(x,R)}(u-k)_+^p\,d\mu\quad\\
   + R^{p-1} \fint_{\dOm \cap B(x,R)} |f|(u-k)_+\,dP_\Om \biggr].
\end{align*}
It follows that
\begin{equation}
  \label{eq:ukr1}
  u(k,r) \le C \biggl(\frac{\mu(A(k,r))}{\mu(B(x,r))}\biggr)^{\frac{\kappa-1}{\kappa p}}
   \biggl(\frac{R}{R-r} u(k,R) + C_f R^{1-1/p} \psi(k,R)^{1/p}\biggr)
\end{equation}
We will now show that $\bigl(\mu(A(k,r))/\mu(B(x,r))\bigr)^{1/p} < C
u(h,R)/(k-h)$ whenever $h<k$. Since $u\ge k$ on $A(k,R)$, we have
\begin{align*}
  (k-h)^p \mu(A(k,r)) & \le \int_{A(k,r)} (u-h)^p\,d\mu \le \int_{A(h,r)} (u-h)^p\,d\mu\\
  & = \mu(B(x,r)) u(h,r)^p \le C \mu(B(x,r)) u(h,R)^p
\end{align*}
as desired.

Using this estimate as well as the inequalities $u(k,R) \le u(h, R)$ and
$\psi(k,R) \le \psi(h, R)$ in \eqref{eq:ukr1} yields that
\begin{equation}
  \label{eq:ukr-uhR}
  u(k,r) \le C \biggl( \frac{u(h,R)}{k-h} \biggr)^{\frac{\kappa-1}{\kappa}}\biggl(\frac{R}{R-r} u(h,R) + C_f R^{1-1/p} \psi(h,R)^{1/p}\biggr).
\end{equation}
Thus we have verified the first of the two inequalities claimed in the lemma.

Let us now establish an analogous inequality for $\psi(k,r)$. Let
$\widetilde{\kappa}>1$ be such that $\widetilde\kappa p = \widetilde{p}$, where
$\widetilde{p}$ is an admissible target exponent for the trace operator, see
Proposition~\ref{pro:trace}. It follows from the H\"older inequality that
\begin{align*}
 \psi(k,r)& = \fint_{\dOm \cap B(x,r)} (u-k)_+\,dP_\Om \\
 & \le \biggl(\fint_{\dOm \cap B(x,r)} (u-k)_+^{\widetilde\kappa p} \,dP_\Om\biggr)^{1/\widetilde\kappa p} \cdot \biggl( \frac{P_\Om(A(k,r)\cap \dOm)}{P_\Om(B(x,r)\cap \dOm)}\biggr)^{1-1/\widetilde\kappa p}\,.
\end{align*}
Then, Proposition~\ref{pro:trace} yields that
\begin{align*}
  \biggl(\int_{\dOm \cap B(x,r)} (u-k)_+^{\widetilde\kappa p} \,dP_\Om\biggr)^{1/\widetilde\kappa p} & \le C r^{1-1/\widetilde{\kappa}p-\aleph} \biggl(\int_{\Om\cap B(x, r)} g_{(u-k)_+}^p \,d\mu\biggr)^{1/p}\\
  & \quad + C P_\Om(\dOm \cap B(x,r))^{1/\widetilde{\kappa}p} \fint_{\Om\cap B(x, r)} (u-k)_+ \,d\mu.
\end{align*}
Combining these two inequalities together with the assumption of co-dimension
$1$ Ahlfors regularity of $P_\Om$ results in
\begin{align}
  \label{eq:psikr-HAkr/muBxr}
  \psi(k,r) & \le C \biggl( \frac{P_\Om(A(k,r)\cap \dOm)}{P_\Om(B(x,r)\cap \dOm)}\biggr)^{1-1/\widetilde\kappa p}
  \\
  & \quad \cdot\biggl( r^{1-\aleph} \mu(B(x,r))^{\frac{\widetilde{\kappa}-1}{\widetilde\kappa p}}
  \biggl(\fint_{\Om\cap B(x, r)} g_{(u-k)_+}^p \,d\mu\biggr)^{1/p} + u(k,r)\biggr). \notag
\end{align}
For an arbitrary $h<k$, we have
\begin{align*}
  (k-h) P_\Om(A(k,r)\cap\dOm) & \le \int_{A(k,r)\cap\dOm} (u-h)\,dP_\Om \\
  &\le \int_{A(h,r)\cap\dOm} (u-h)\,dP_\Om \le P_\Om(B(x,r)\cap\dOm) \psi(h,r).
\end{align*}
Applying this inequality together with \eqref{eq:g(u-k)+est} to
\eqref{eq:psikr-HAkr/muBxr} yields that
\begin{align*}
  & \psi(k,r) \\
  & \le C \biggl( \frac{\psi(h,r)}{k-h} \biggr)^{\frac{\widetilde{\kappa}p-1}{\widetilde\kappa p}}
   \biggl( r^{1-\aleph} \mu(B(x,r))^{\frac{\widetilde{\kappa}-1}{\widetilde\kappa p}}
   \biggl(\fint_{\Om\cap B(x, r)} g_{(u-k)_+}^p \,d\mu\biggr)^{1/p} + u(k,r)\biggr) \\
  & \le C \biggl( \frac{\psi(h,R)}{k-h} \biggr)^{\frac{\widetilde{\kappa}p-1}{\widetilde\kappa p}}
  \biggl[ r^{1-\aleph} \mu(B(x,r))^{\frac{\widetilde{\kappa}-1}{\widetilde\kappa p}} \biggl(\frac{u(k,R)}{R-r}
  + \frac{(C_f \psi(k,R))^{1/p}}{R^{1/p}}\biggr)+ u(k,R)\biggr]\\
  & \le C \biggl( \frac{\psi(h,R)}{k-h} \biggr)^{\frac{\widetilde{\kappa}p-1}{\widetilde\kappa p}} \biggl[
    \biggl(1+\frac{R^{1-\aleph}}{R-r}\biggr) u(h,R) + C_f R^{1-1/p -\aleph}\psi(h,R)^{1/p}
  \biggr],
\end{align*}
where the crude estimate $\mu(B(x,r)) \le \mu(\Om)$ was used in the last line.
Since $R-r\le R/2\le R_0/2$, and since $0<1-\aleph<1$, the desired inequality
for $\psi$ follows.

If $\mu$ happens to be Ahlfors $s$-regular at scale $r_0>0$, then a finer
estimate $\mu(B(x,r)) \le C r^s$ is to be used above. Since
$\aleph=s(\frac{1}{p}-\frac{1}{\widetilde{p}})$ and $\widetilde{p} =
\widetilde{\kappa} p$, we have
\[
  r^{-\aleph} \mu(B(x,r))^{\frac{\widetilde{\kappa}-1}{\widetilde\kappa p}}
  \le C r^{s (\frac{1}{\widetilde{\kappa} p} - \frac{1}{p})} r^{s \frac{\widetilde{\kappa}-1}{\widetilde\kappa p}} = C\,.
\]
Then, it follows from the penultimate line of the estimate of $\psi(k,r)$ above
that
\begin{align*}
  \psi(k,r)
  & \le C \biggl( \frac{\psi(h,R)}{k-h} \biggr)^{\frac{\widetilde{\kappa}p-1}{\widetilde\kappa p}}
  \biggl[ r \biggl(\frac{u(k,R)}{R-r}
  + \frac{(C_f \psi(k,R))^{1/p}}{R^{1/p}}\biggr)+ u(k,R)\biggr]\\
  & \le C \biggl( \frac{\psi(h,R)}{k-h} \biggr)^{\frac{\widetilde{\kappa}p-1}{\widetilde\kappa p}} \biggl[
    \biggl(1+\frac{R}{R-r}\biggr) u(h,R) + C_f R^{1-1/p}\psi(h,R)^{1/p}
  \biggr]\,.
\end{align*}
Again noting that $R-r\le R/2$, we obtain the
inequality~\eqref{eq:Ahlfors-strong}.
\end{proof}
We are now ready to prove the main theorem of this section. Recall that the
minimizer $u$ necessarily belongs to $L^1(\Om)$ and its trace belongs to
$L^1(\partial\Om, P_\Om)$. The boundedness estimates we obtain in the proof
indicate that the bound on $u$ is determined by its trace's average value on
the boundary of $\Om$ with respect to the measure $P_\Om$ as well as on the
average of $u$ on the ball, and on the bound on $f$ on the boundary of $\Om$.
This is in contrast to the local boundedness estimates of~\cite{KS} for
$p$-energy minimizers in the interior of $\Om$, where the bound is determined
by the average value of $u$ alone.
\begin{proof}[Proof of Theorem~\ref{thm:bounded}]
In order to prove that $u$ is bounded from above near the boundary, it suffices
to show that for a fixed $R>0$ with $R<R_0/4$ and $k_0 \in \RR$ we can find $d
\ge 0$ such that $u(k_0+d,R/2) = 0$, where $u(k,r)$ is as in
Lemma~\ref{lem:decay-est}.

If $u(k_0,R) = 0$, then we immediately obtain the upper bound that $u\le k_0$
in $B(x,R)$. In what follows, suppose that $u(k_0,R)> 0$.

Let $r_n = (1+2^{-n})\cdot R/2$ and $k_n = k_0+d(1-2^{-n})$, where the precise
value of $d>0$ will be determined later. Setting $h=k_n$, $k=k_{n+1}$, $\rho =
r_n$, and $r= r_{n+1}$ in \eqref{eq:ukr-uhR,psikr-psihR} yields that
\begin{align}
\notag
  u(k_{n+1},r_{n+1}) &\le C \biggl( \frac{u(k_n,r_n)}{2^{-n-1}d} \biggr)^{\alpha}\biggl(\frac{1+2^{-n}}{2^{-n-1}} u(k_n,r_n) + C_f r_n^{1-1/p} \psi(k_n,r_n)^{1/p}\biggr)\\
  &\le C_{f,R} \frac{2^{n(\alpha + 1)}}{d^{\alpha}} \bigl(u(k_n, r_n)^{1+\alpha} + u(k_n, r_n)^{\alpha} \psi(k_n,r_n)^{1/p}\bigr)
  \label{eq:ukn+1rn+1}
\end{align}
and analogously
\begin{equation}
\label{eq:psikn+1rn+1}
  \psi(k_{n+1},r_{n+1}) \le C_{f,R} \frac{2^{n(\beta + 1)}}{d^{\beta}} \bigl(u(k_n, r_n) \psi(k_n,r_n)^{\beta} + \psi(k_n, r_n)^{\beta+1/p}\bigr),
\end{equation}
where $C_{f,R} = C\cdot (1+C_f R^{1-1/p}+R^{-\aleph} + C_f R^{1-1/p-\aleph})$. By induction, we will show that
\begin{equation}
  u(k_n, r_n) \le 2^{-\sigma n} u(k_0,R)\quad\text{ and }\quad\psi(k_n, r_n) \le 2^{-\tau n} \psi(k_0,R)
\label{eq:ukn,psikn-ind}
\end{equation}
for a suitable choice of positive constants $\sigma$, $\tau$, and $d$. In such
a case, we will have $u(k_0+d,R/2) = \lim_{n\to\infty} u(k_n, r_n) = 0$.
Observe that both inequalities in \eqref{eq:ukn,psikn-ind} are satisfied for
$n=0$.

If $\psi(k_0, R) = 0$, then the second inequality in \eqref{eq:ukn,psikn-ind}
is vacuously satisfied. If $\psi(k_0, R) \neq 0$, then \eqref{eq:psikn+1rn+1}
together with \eqref{eq:ukn,psikn-ind} lead to
\begin{align*}
\psi(k_{n+1}&, r_{n+1}) \le \frac{\psi(k_0,R)}{2^{\tau (n+1)}} \cdot \frac{2^{\tau (n+1)}}{\psi(k_0,R)} \\
& \quad \cdot C_{f,R} \frac{2^{n(\beta + 1)}}{d^{\beta}} \biggl[\frac{u(k_0, R)}{2^{\sigma n}} \biggl(\frac{\psi(k_0, R)}{2^{\tau n}}\biggr)^{\beta} + \biggl(\frac{\psi(k_0, R)}{2^{\tau n}}\biggr)^{\beta+1/p}\biggr] \\
& \le \frac{\psi(k_0,R)}{2^{\tau (n+1)}} \cdot \frac{C_{f,R}}{d^{\beta} \psi(k_0,R)^{1-\beta}} 2^{\tau+n(\tau + \beta + 1 - \tau\beta)}\biggl(\frac{u(k_0,R)}{2^{\sigma n}} + \frac{\psi(k_0,R)^{1/p}}{2^{\tau n/p}}\biggr).
\end{align*}
Thus, if \eqref{eq:ukn,psikn-ind} is to be satisfied when $\psi(k_0, R) \neq
0$, we need
\begin{equation}
  \label{eq:sigmatau1}
  \tau + \beta + 1 - \tau\beta - \sigma \le 0\quad \text{and}\quad \tau + \beta + 1 - \tau\beta - \frac{\tau}{p} \le 0
\end{equation}
as well as
\begin{equation}
  \label{eq:d-est1}
  d \ge \biggl(\frac{C_{f,R} 2^\tau \bigl(u(k_0,R) + \psi(k_0,R)^{1/p}\bigr)}{\psi(k_0,R)^{1-\beta}}\biggr)^{1/\beta}.
\end{equation}
Analogously, inequalities \eqref{eq:ukn+1rn+1} and \eqref{eq:ukn,psikn-ind}
provide us with the estimate
\begin{align*}
  u(k_{n+1}, r_{n+1}) & \le \frac{u(k_0,R)}{2^{\sigma (n+1)}} \\
  & \quad \cdot \frac{C_{f,R}}{d^{\alpha} u(k_0,R)^{1-\alpha}} 2^{\sigma+n(\sigma + \alpha + 1 - \sigma\alpha)}\biggl(\frac{u(k_0,R)}{2^{\sigma n}} + \frac{\psi(k_0,R)^{1/p}}{2^{\tau n/p}}\biggr).
\end{align*}
Therefore, we need
\begin{equation}
  \label{eq:sigmatau2}
  \alpha + 1 - \sigma \alpha \le 0 \quad \text{and} \quad
  \sigma + \alpha + 1 - \sigma \alpha - \frac{\tau}{p} \le 0
\end{equation}
as well as
\begin{equation}
  \label{eq:d-est2}
  d \ge \biggl(\frac{C_{f,R}2^\sigma \bigl(u(k_0,R)
  + \psi(k_0,R)^{1/p}\bigr)}{u(k_0,R)^{1-\alpha}} \biggr)^{1/\alpha}.
\end{equation}
Simplifying \eqref{eq:sigmatau1} and \eqref{eq:sigmatau2} yields
\begin{align*}
 \max\biggl\{1+ \frac{1}{\alpha}\,, \tau(1-\beta)+ 1+\beta \biggr\} & \le \sigma
 \le \frac{\frac{\tau}{p} -(1+\alpha)}{1-\alpha}\quad \text{and}\quad
\tau \ge \frac{\beta + 1}{\beta + \frac1p -1}.
\end{align*}
Recall that $\alpha = 1 - \frac{1}{\kappa}$ and $\beta = 1-
\frac{1}{\widetilde{\kappa} p}$, where $\kappa>1$ is chosen such that
$N^{1,p}(\Om) \subset L^{\kappa p}(\Om)$ while $\widetilde{\kappa}>1$ is chosen
such that the trace operator maps $N^{1,p}(\Om)$ into
$L^{\widetilde{\kappa}p}(\dOm)$. Choosing
\[
  \tau \ge \max\biggl\{\frac{2 \widetilde{\kappa} p -1}{\widetilde{\kappa}-1}\,,
  p (\kappa-1),
  \frac{2p+2\kappa -1 - 1/\widetilde{\kappa}}{\kappa - 1/\widetilde{\kappa}}
  \biggr\}
\]
will allow us to find $\sigma$ so that both \eqref{eq:sigmatau1} and
\eqref{eq:sigmatau2} are fulfilled, which will then enable us to use
\eqref{eq:d-est1} and \eqref{eq:d-est2} to find a sufficiently big value of
$d$.

For such a constant $d$, we have
\[
  0 = u\Bigl(k_0+d, \frac{R}{2}\Bigr) = \biggl(\fint_{\Om\cap B(x,R/2)} (u-k_0-d)^p_+\,d\mu\biggr)^{1/p},
\]
which shows that $u \le k_0+d$ $\mu$-a.e.\@ in $B(x,R/2)$. Analogously, we have
the trace $Tu \le k_0+d$ $P_\Om$-a.e.\@ in $\dOm \cap B(x,R/2)$. Running the
argument once more with $u$ and $f$ replaced by $-u$ and $-f$, respectively, we
obtain that $u \in L^\infty(\Om_R)$ and $Tu \in L^\infty(\dOm)$, where $\Om_R =
\{z \in \Om: \dist(z,\dOm) < R/2\}$.

Letting $k_0=0$ yields the desired conclusion.
\end{proof}
\section{Further boundary regularity}
In PDE literature, the part of the boundary where the Neumann data $f$ vanishes
is called the natural boundary. If $x\in\dOm$ and $r>0$ such that $f=0$ on
$\dOm\cap B(x,r)$, then
\[
 \int_{\Om\cap B(x,r)}g_u^p\, d\mu\le \int_{\Om\cap B(x,r)}g_{u+\varphi}^p\, d\mu
\]
for every $\varphi\in N^{1,p}(X)$ with compact support in $B(x,r)$, i.e., $u$
is $p$-harmonic in $\Om\cup(\dOm\cap B(x,r))$. Thus, given our standing
assumptions on $\Om$, the results of~\cite{KS} apply to $u$ on
$B(x,r)\cap\overline{\Om}$, to yield that $u$ is locally H\"older continuous in
$B(x,r)\cap\overline{\Om}$. We have so far no boundary H\"older continuity of
$u$ at other parts of $\dOm$. In the Euclidean setting, we know from the work
of~\cite{ADN1, ADN2, KLS, Tay} that if $\Om$ is a bounded Euclidean domain of
class $C^1$, and the boundary data $f$ is H\"older continuous, then $u$ is
H\"older continuous at $\dOm$. On the other hand, we obtain partial regularity
results for $u$ near sets of positivity of $f$ (and correspondingly, sets of
negativity of $f$) in this section using the results from~\cite{KinMa1, KinMa2,
BBbook} on nonlinear potential theory on metric measure spaces. These will
allow us to prove continuity of $u$ up to the boundary on open subsets of
positivity (or negativity) of $f$ for values of $p$ close to $1$ or close to
$s$ in Section~\ref{sec:cont}.
\begin{definition}
Let $(Y,d_Y,\mu_Y)$ be a metric measure space. A function $v$ on an open set
$A\subset Y$ is a \emph{$p$-subminimizer} if
\[
 \int_A g_v^p\, d\mu_Y\le \int_A g_{v+\varphi}^p\, d\mu_Y
\]
for every non-positive $\varphi\in N^{1,p}(Y)$ that is compactly supported in
$A$.
\end{definition}
The notion of subminimizers in the metric setting is extensively studied; a
non-exhaustive listing of papers about subminimizers in the metric setting
is~\cite{S1, KinMa1, KinMa2, BAnd, BJan, BBP}. The book~\cite{BBbook} contains
a nice discussion of nonlinear potential theory in metric setting.

It is known that if $\mu_Y$ is doubling, $Y$ is complete, and supports a
$p$-Poincar\'e inequality, then subminimizers are $p$-finely continuous in $A$
(see~\cite{BJan} or~\cite[Theorem~11.38]{BBbook}) and are upper semicontinuous
in $A$ (see~\cite{KinMa1} or~\cite[Theorem~8.22]{BBbook}). Recall that a
function is \emph{$p$-finely continuous} at $z\in A$ if it is continuous with
respect to the $p$-fine topology on $Y$. Here, a set $U\subset Y$ is
\emph{$p$-finely open} if $Y\setminus U$ is $p$-finely thin at each $x\in U$,
that is,
\begin{equation}\label{eq:p-fat}
\int_0^1\left(\frac{\text{cap}_p(B(x,\rho)\setminus U, B(x,2\rho))}
   {\text{cap}_p(B(x,\rho), B(x,2\rho))}\right)^{1/(p-1)}\, \frac{d\rho}{\rho}<\infty.
\end{equation}
Here, for $E\subset B(x,\rho)$, the quantity $\text{cap}_p(E, B(x,2\rho))$ is
the relative variational $p$-capacity of $E$ with respect to $B(x,2\rho)$ as
given in Definition~\ref{def:rel-cap}; see~\cite[Section~11.6]{BBbook}.
\begin{proposition}\label{prop:semicont}
Let $x\in\dOm$ and $r>0$ such that $f\ge0$ on $B(x,r)\cap\dOm$. Then $u$ is a
$p$-subminimizer on $B(x,r)\cap\overline{\Om}$, and hence is upper
semicontinuous at $x$, that is,
\[
u(x)\ge \limsup_{B(x,r)\cap\overline{\Om}\ni y\to x}u(y),
\]
and $u$ is $p$-finely continuous in $B(x,r)\cap\overline{\Om}$.
\end{proposition}
\begin{proof}
From our standing hypothesis that $\Om$ supports a $p$-Poincar\'e inequality
and that the restriction of $\mu$ to $\Om$ satisfies~\eqref{density}, we know
that $\overline{\Om}$, equipped with the inherited metric and the restriction
of $\mu$ to $\overline\Om$ is doubling and supports a $p$-Poincar\'e
inequality. Hence the results regarding $p$-subharmonic functions mentioned
above would yield the desired conclusions regarding $u$ provided we demonstrate
that $u$ is a $p$-subminimizer on $B(x,r)\cap\overline{\Om}$.

To this end, let $\varphi\in N^{1,p}(\overline{\Om})$ be a non-positive
function such that $\varphi=0$ on $\overline{\Om}\setminus B(x,r)$. With
$u+\varphi$ as a competitor, we know that $I(u)\le I(u+\varphi)$, that is,
\[
\int_\Om g_u^p\, d\mu+\int_{\dOm}uf\, dP_\Om
  \le \int_\Om g_{u+\varphi}^p\, d\mu+\int_{\dOm}(u+\varphi)f\, dP_\Om.
\]
It follows from $g_u=g_{u+\varphi}$ $\mu$-a.e.~in $\overline{\Om}\setminus
B(x,r)$ and from $\mu(\dOm)=0$ that
\[
\int_{B(x,r)\cap\overline{\Om}}g_u^p\, d\mu
\le \int_{B(x,r)\cap\overline{\Om}}g_{u+\varphi}^p\, d\mu
 +\int_{\dOm\cap B(x,r)}\varphi f\, dP_\Om.
\]
Because $f\ge 0$ on $\dOm\cap B(x,r)$ and $\varphi\le 0$ there, it follows that
\[
\int_{B(x,r)\cap\overline{\Om}}g_u^p\, d\mu
\le \int_{B(x,r)\cap\overline{\Om}}g_{u+\varphi}^p\, d\mu
\]
as desired.
\end{proof}
We next show that if $u$ is constant in a neighborhood of a point in the
boundary, then that point belongs to the natural boundary (that is, $f$
vanishes in a relative neighborhood of that point).
\begin{proposition}
Let $u$ be a $p$-harmonic solution to the Neumann boundary value problem on
$\Om$ with continuous boundary data $f$, and if $x\in\dOm$ and $r>0$ such that
$u$ is constant on $B(x,r)\cap\Om$, then $f=0$ on $B(x,r/2)$.
\end{proposition}
\begin{proof}
It suffices to show that for each such $x$ and $r>0$ we have $f(x)=0$. Suppose
that $f(x)> 0$ (by replacing $f$ with $-f$ and $u$ with $-u$ if necessary).
Then for sufficiently small $r>0$ we have in addition to $u$ being constant on
$B(x,r)\cap\Om$ that $f>0$ on $B(x,r)\cap\dOm$.

Let the constant value of $u$ on $B(x,r)\cap\Om$ be $M$. For $k\in\RR$ with
$k<M$, with the choice of $v=u-\eta_{r/2,r}(u-k)_+$ as
in~\eqref{eq:gv_estimate} that
\begin{align*}
M\int_{\dOm\cap B(x,r)}f\, dP_\Om&=
\int_{\Om\cap B(x,r)}g_u^p\, d\mu+\int_{\dOm\cap B(x,r)}uf\, d\mu\\
&\le \int_{\Om\cap B(x,r)}g_v^p\, d\mu+\int_{\dOm\cap B(x,r)}vf\, dP_\Om.
\end{align*}

Since $g_v\le (1-\eta)g_{u-k}+\tfrac{2}{r}(u-k)_+=\tfrac{2}{r}(u-k)_+$ on
$B(x,r)\setminus B(x,r/2)$ $\mu$-a.e., it follows that
\[
\int_{\dOm\cap B(x,r)}M\, f\, dP_\Om \le \frac{2^p}{r^p}(M-k)^p\mu(B(x,r)\setminus B(x,r/2))
  +\int_{\dOm\cap B(x,r)}vf\, dP_\Om.
\]
Thus
\begin{align*}
\int_{\dOm\cap B(x,r)}(M-v)f\, dP_\Om
&\le \frac{2^p(M-k)^p}{r^p}\mu([B(x,r)\setminus B(x,r/2)]\cap\Om)\\
&\le \frac{2^p}{r^p}\mu(B(x,r))(M-k)^p.
\end{align*}
Since $v\le M$, and as $\lim_{k\to M^-}\frac{M-v}{M-k}=1$ on
$B(x,r/2)\cap\dOm$, we have
\begin{align*}
\int_{\dOm\cap B(x,r/2)}\frac{M-v}{M-k}\, f\, dP_\Om
  &\le \frac{2^p(M-k)^{p-1}}{r^p}\mu([B(x,r)\setminus B(x,r/2)]\cap\Om)\\
	&\le \frac{2^p}{r^p}\mu(B(x,r))(M-k)^{p-1},
\end{align*}
and letting $k\to M^-$ we obtain
\[
0\le \int_{\dOm\cap B(x,r/2)}\, f\, dP_\Om\le 0.
\]
Here, we used the fact that $p>1$.
\end{proof}
As a consequence of the above proposition, we know that if the boundary data
$f$ is not constant (equivalently, not the zero function), then $u$ is not
constant on $\Om$. This agrees with our intuitive understanding of the boundary
data $f$ controlling the ``outer normal derivative" of $u$ at $\partial\Om$ ---
if the derivative cannot vanish on the boundary, then the function cannot be
constant. This is in spite of the fact that we do not have analogous
differential equation in the metric setting.
\section{Boundary continuity for $p$ close to $1$ or the natural dimension $s$ when $\mu$ is Ahlfors $s$-regular at small scales}
\label{sec:cont}
In this section we need the strong version~\eqref{eq:Ahlfors-strong}. We
therefore assume from now on that $\mu$ is Ahlfors $s$-regular at scale
$r_0>0$.

Recall that the exponents $\alpha$ and $\beta$ used in Section~5 to prove
boundedness of the solution $u$ depend on $p$ and the exponent $s$
from~\eqref{eq:mu-dimension}. When $p$ is close to either $1$ or $s$, then it
is possible to find values of $\alpha$ and $\beta$ such that
\begin{equation}\label{eq:condition-on-p}
\alpha+\frac1p-1>0\qquad \text{and}\qquad \beta+\frac1p-1>0.
\end{equation}
The above conditions are satisfied whenever $p^2 - sp + s>0$. In particular,
they allow for all $p>1$ if the dimension $s<4$. In this section we will show
that when $p$ satisfies~\eqref{eq:condition-on-p}, the function $u$ is
continuous up to the boundary of $\Om$.
\begin{theorem}\label{thm:Main2}
Suppose that $\mu$ is Ahlfors $s$-regular at scale $r_0>0$. Under the standard
assumptions on $\Om$ and $\mu$, if $f:\dOm\to\RR$ is a bounded Borel measurable
function on $\dOm$,  $x\in\dOm$, and $r_0>0$ such that $f\ge 0$ on
$B(x,r_0)\cap\dOm$ or $f\le 0$ on $B(x,r_0)\cap\dOm$, then $u$ is continuous at
$x$ relative to $\overline{\Om}$.
\end{theorem}
\begin{proof}
Without loss of generality we may assume that $f\le 0$ on $\dOm\cap B(x,r_0)$,
for if $f\ge 0$ at each point in $\dOm\cap B(x,r_0)$ then we apply the
following analysis to $-u$, which is a solution for the boundary data $-f$.

Suppose that $u$ is not continuous at $x$.  For $R>0$ we set
\[
 M(R)\coloneq \sup_{y\in B(x,R)\cap\Om}u(y)\qquad\text{and}\qquad m(R)\coloneq \inf_{y\in B(x,R)\cap\Om}u(y).
\]
Then by assumption we have that $\lim_{R\to0^+}M(R)\eqcolon M>\lim_{R\to0^+}m(r)\eqcolon m$.

For $0<R<\min\{1,r_0\}$, $k_0\in\RR$ with $k_0<M(R)$, and for $n\in\NN$ we set
$r_n=(1+2^{-n})R/2$ and $k_n=k_0+d(1-2^{-n})$, where we want to choose $d>0$
such that we have $u\le k_0+d$ on $B(x,R/2)\cap\Om$. In other words, we repeat
the proof of boundedness of $u$, but now we modify the choice of $d$ by
modifying~\eqref{eq:ukn,psikn-ind}. As in Lemma~\ref{lem:decay-est}, we set
\[
u(k, r) =\biggl(\fint_{\Om\cap B(x,r)} (u-k)_+^p\,d\mu\biggr)^{1/p}\quad \text{and}\quad
\psi(k, R) =\fint_{\dOm\cap B(x,R)} (u-k)_+\,dP_\Om.
\]
Suppose that $k_0\in\RR$ such that
\[
  \frac{\mu(A(k_0,R))}{\mu(B(x,R)\cap\Om)}\le \frac{1}{(4D)^p},
\]
then we wish to show that there exist $\sigma,\tau>0$ such that for each
$n\in\NN$,
\begin{equation}\label{eq:mod-est}
u(k_n,r_n)\le \frac{2^{-\sigma n}(M(R)-k_0)}{4D}\qquad \text{and}\qquad
\psi(k_n,r_n)\le 2^{-\tau n}(M(R)-k_0).
\end{equation}
Here in the above, we just replaced $4C[1+C_f]$ with $C$, and we remind the
reader that we are not particularly concerned with the precise value of the
constants $C$ as long as they are independent of $R$. This holds when $n=0$.
Suppose we know that the above holds for some non-negative integer $n$. Observe
that by Theorem~\ref{thm:bounded} we have $|M(R)|<\infty$ and $|m(R)|<\infty$.
By~\eqref{eq:ukr-uhR,psikr-psihR} of Lemma~\ref{lem:decay-est} we have
\begin{align*}
&u(k_{n+1},r_{n+1})\le C\left[\frac{u(k_n,r_n)}{k_{n+1}-k_n}\right]^\alpha
   \left[\frac{r_n}{r_n-r_{n+1}}u(k_n,r_n)+r_n^{1-1/p}\psi(k_n,r_n)^{1/p}\right]\\
   &\qquad\le C\left[\frac{2^{-n(\sigma-1)}(M(R)-k_0)}{4Dd}\right]^\alpha
    \left[\frac{(M(R)-k_0)}{4D\, 2^{n(\sigma-1)}}
      +\frac{R^{1-1/p}(M(R)-k_0)^{1/p}}{2^{\tau n/p}}\right]
\end{align*}
and by~\eqref{eq:Ahlfors-strong},
\begin{align*}
&\psi(k_{n+1},r_{n+1})\le C\left[\frac{\psi(k_n,r_n)}{k_{n+1}-k_n}\right]^\beta
   \left[\frac{r_n}{r_n-r_{n+1}}u(k_n,r_n)+r_n^{1-1/p}\psi(k_n,r_n)^{1/p}\right]\\
   &\qquad\le C\left[\frac{2^{-n(\tau-1)}(M(R)-k_0)}{d}\right]^\beta
    \left[\frac{(M(R)-k_0)}{4D\, 2^{n(\sigma-1)}}
      +\frac{R^{1-1/p}(M(R)-k_0)^{1/p}}{2^{\tau n/p}}\right].
\end{align*}
Therefore~\eqref{eq:mod-est} would hold for $n+1$ if we can ensure that
\begin{align*}
C\left[\frac{2^{-n(\sigma-1)}(M(R)-k_0)}{4Dd}\right]^\alpha &
    \left[\frac{(M(R)-k_0)}{4D\, 2^{n(\sigma-1)}}
     +\frac{R^{1-1/p}(M(R)-k_0)^{1/p}}{2^{\tau n/p}}\right]\\
    &\le \frac{2^{-\sigma(n+1)}(M(R)-k_0)}{4D},
\end{align*}
and
\begin{align*}
C\left[\frac{2^{-n(\tau-1)}(M(R)-k_0)}{d}\right]^\beta&
    \left[\frac{(M(R)-k_0)}{4D\, 2^{n(\sigma-1)}}
    +\frac{R^{1-1/p}(M(R)-k_0)^{1/p}}{2^{\tau n/p}}\right]\\
     &\le 2^{-\tau(n+1)}(M(R)-k_0).
\end{align*}
The above two inequalities are satisfied if we can guarantee that
\begin{align}\label{eq:Conditions-sigma-tau-d}
\sigma&\ge \frac{\alpha+1}{\alpha},\notag\\ \tau&\ge p[\sigma(1-\alpha)+\alpha],\notag\\
\tau&\ge \frac{\beta}{\beta+\frac1p-1},\notag\\ \tau&\le \frac{\sigma-(1+\beta)}{1-\beta},\notag \\
d&\ge \max\{C^{1/\alpha},C^{1/\beta}\}\frac{(M(R)-k_0)}{4D}, \\
d&\ge C^{1/\alpha}\left[(M(R)-k_0)^{\alpha+\frac1p-1}R^{1-1/p}(4D)^{1-\alpha}\right]^{1/\alpha},\notag\\
d&\ge C^{1/\beta}\left[(M(R)-k_0)^{\beta+\frac1p-1}R^{1-1/p}\right]^{1/\beta}.\notag
\end{align}
In the above, we choose $D>1$ such that
\[
D\ge \max\{C^{1/\alpha},C^{1/\beta}\}.
\]
Given the assumptions~\eqref{eq:condition-on-p} on $p$, the above are
guaranteed by the choices of $\sigma$, $\tau$, and $d$ such that
\begin{align}\label{eq:choices-sigma-tau}
\max\bigg\lbrace\frac{\alpha+1}{\alpha}, 1+\beta+\frac{\beta(1-\beta)}{\beta+\frac1p-1},
  \frac{1+\beta}{1-p(1-\alpha)(1-\beta)}\bigg\rbrace&= \sigma,\notag\\
\max\bigg\lbrace\frac{\beta}{\beta+\frac1p-1}, p[\sigma-(\sigma-1)\alpha]\bigg\rbrace&\le \tau\le \frac{\sigma-(1+\beta)}{1-\beta},
\end{align}
and it suffices to choose $d$ as follows:
\begin{align}\label{eq:choice-d}
\max\biggl\lbrace \frac{(M(R)-k_0)}{4}, C\Bigl[R^{1-1/p}(M(R)&-k_0)^{\alpha+\frac1p-1}\Bigr]^{1/\alpha},\notag\\
  C\Bigl[R^{1-1/p}&(M(R)-k_0)^{\beta+\frac1p-1}\Bigr]^{1/\beta}\biggr\rbrace = d.
\end{align}
The above choice of $\tau$ is possible because of the
assumptions~\eqref{eq:condition-on-p} on $p$. Thus given $k_0<M(R)$ we have the
above choice of $d$, $\sigma$, and $\tau$ such that, by letting $n\to\infty$
in~\eqref{eq:mod-est}, we can conclude that $u\le k_0+d$ on $B(x,R/2)\cap\Om$.

We only consider $0<R<\max\{1,r_0\}$ for which
\[
0<M-m\le M(R)-m(R)\le 2(M-m).
\]

Finally, for $\nu\in\NN$ set $\kappa_\nu=M(R)-2^{-\nu-1}(M(R)-m(R))$. By
Proposition~\ref{prop:semicont}, $u$ is lower semicontinuous at $x$, and so
$m=Tu(x)$. Furthermore, by this proposition we have that $u$ is finely
continuous at $x$, and so by~\eqref{eq:p-fat} together
with~\cite[Proposition~6.16]{BBbook} (see~\eqref{eq:rel-cap-Ball}), $\lim_{R\to
0^+}\frac{\mu(A(\kappa_\nu,R))}{\mu(B(x,R)\cap\Om)}=0$ for sufficiently large
$\nu$. Fix such $\nu\ge 3$ and we further restrict $R$ for which
\[
\frac{\mu(A(\kappa_\nu,r))}{\mu(B(x,r)\cap\Om)}\le \frac{1}{(4D)^p}
\]
whenever $0<r\le R$. Then by the above, with $\kappa_\nu$ playing the role of
$k_0$, we have $M(R)-\kappa_\nu=2^{-(\nu+1)}(M(R)-m(R))$, and so
\begin{align*}
M(R/2)-m(R/2)\le \kappa_\nu-m(R/2)+d
    &\le \kappa_\nu-m(R)+d\\
    &= [1-2^{-(\nu+1)}](M(R)-m(R))+d.
\end{align*}
We further restrict $R$ so that
\begin{equation}\label{eq:R-rest1}
\lambda_1\coloneq 1-2^{-(\nu+1)}+\frac{CR^{(1-1/p)/\alpha}2^{-(\nu+1)\widehat{\alpha}}}{(M-m)^{1-\widehat{\alpha}}}<1
\end{equation}
and
\begin{equation}\label{eq:R-rest2}
\lambda_2\coloneq 1-2^{-(\nu+1)}+\frac{CR^{(1-1/p)/\beta}2^{-(\nu+1)\widehat{\beta}}}{(M-m)^{1-\widehat{\beta}}}<1.
\end{equation}
Here,
\[
  \widehat{\alpha}=\left[\alpha+\frac1p-1\right]/\alpha<1,\qquad\widehat{\beta}=\left[\beta+\frac1p-1\right]/\beta<1.
\]
If $d=\tfrac14(M(R)-\kappa_\nu)$, then  we see by the choice of $\nu\ge 3$ as
outlined above that
\begin{align}\label{eq:Est1}
M(R/2)-m(R/2)&\le [1-2^{-(\nu+1)}](M(R)-m(R))+\frac{2^{-(\nu+1)}}{4}(M(R)-m(R))\notag\\
  &\le [1-2^{-(\nu+2)}](M(R)-m(R)).
\end{align}
If $d=C\left[R^{1-1/p}(M(R)-k_0)^{\alpha+\frac1p-1}\right]^{1/\alpha}$, then by
the restriction~\eqref{eq:R-rest1} we have from $M(R)-m(R)\approx M-m$ that
\begin{align}\label{eq:Est2}
M(R/2)-m(R/2)&\le [1-2^{-(\nu+1)}](M(R)-m(R))\notag\\
 &\qquad\quad +\frac{CR^{(1-1/p)/\alpha}2^{-(\nu+1)\widehat{\alpha}}}{(M-m)^{1-\widehat{\alpha}}}(M(R)-m(R))\notag\\
 &\le \lambda_1 (M(R)-m(R)).
\end{align}
If $d=C\left[R^{1-1/p}(M(R)-k_0)^{\beta+\frac1p-1}\right]^{1/\beta}$, then
similarly we obtain
\begin{equation}\label{eq:Est3}
M(R/2)-m(R/2)\le \lambda_2(M(R)-m(R)).
\end{equation}
Combining~\eqref{eq:Est1}, \eqref{eq:Est2}, and~\eqref{eq:Est3}, setting
\[
\lambda=\max\{1-2^{-(\nu+2)}, \lambda_1,\lambda_2\},
\]
and noting that $0<\lambda<1$, we obtain in all three cases that for all small
$R>0$,
\[
M(R/2)-m(R/2)\le \lambda (M(R)-m(R)).
\]
An iterated application of the above tells us that
\[
M(r)-m(r)\le 2^{1+\theta_0}\left(\frac{r}{R}\right)^{\theta_0}[M-m]
\]
for all $0<r<R$, where $\theta_0=\log_2(1/\lambda)$. It follows that $u$ must
have $\theta_0$-H\"older continuous decay to $Tu(x)$ at $x$, which contradicts
our assumption that $u$ is \emph{not} continuous at $x$.

Thus we conclude that $u$ must be continuous at $x$ from $\Om$, that is,
\[
\lim_{\Om\ni y\to x}u(y)=Tu(x).
\]
This holds for each $x\in\dOm\cap B(y,r)$ on which $f$ does not change sign.
Since $Tu$ is the trace of $u$ on $\dOm$, it follows that $u$ is continuous at
$x$ relative to $\overline{\Om}$.
\end{proof}
Note that the above proof does not permit us to conclude that $u$ must be
H\"older continuous at the boundary point $x$. From the work of~\cite{Cr, KLS}
we know that in the Euclidean setting, with $\Om$ a bounded smooth domain, $u$
is H\"older continuous at the boundary. As far as we know, this remains open in
the metric setting.

The above proof does not permit us to draw any conclusions at boundary points
where $f$ changes sign. On the other hand, an analysis of the proof above shows
that if there is some $\xi\in [m(R), M(R)]$ for which
\[
 \lim_{r\to0^+} \fint_{B(x,r)\cap\Om}|u-\xi|\, d\mu =0,
\]
then when $\lim_{r\to0^+}M(r)=M>m=\lim_{r\to0^+}m(r)$, we must have either
\[
 \lim_{r\to 0^+}\frac{\mu(\{u>(m+M)/2\}\cap B(x,r)\cap\Om)}{\mu(B(x,r)\cap\Om)}=0
\]
or
\[
 \lim_{r\to 0^+}\frac{\mu(\{u<(m+M)/2\}\cap B(x,r)\cap\Om)}{\mu(B(x,r)\cap\Om)}=0.
\]
By considering $u$ in the first case and $-u$ in the second case, for
sufficiently large $\nu$, with $\kappa_\nu=M(R)-2^{-(\nu+1)}[M(R)-m(R)]$ we
have
\[
\lim_{r\to0^+}\frac{\mu(A(\kappa_\nu,r))}{\mu(B(x,r)\cap\Om)}=0,
\]
and so the proof of Theorem~\ref{thm:Main2} will show that $u$ has to be
continuous at $x$. Note that here we will obtain that $\xi=Tu(x)$. By the
definition of the trace function $Tu$, we have
\[
 \lim_{r\to0^+} \fint_{B(x,r)\cap\Om}|u-Tu(x)|\, d\mu =0 \quad\text{for $\mathcal{H}$-a.e.~$x\in\dOm$.}
\]
Thus, we have the following theorem.
\begin{theorem}\label{thm:Main3}
Under the standard assumptions on $\Om$ and $\mu$, if $f:\dOm\to\RR$ is a
bounded Borel measurable function on $\dOm$,  then for $\mathcal{H}$-almost
every $x\in\dOm$, $u$ is continuous at $x$ relative to $\overline{\Om}$.
\end{theorem}
\end{document}